\documentclass[reqno, 12pt]{amsart}
\pdfoutput=1
\makeatletter
\let\origsection=\section \def\section{\@ifstar{\origsection*}{\mysection}}
\def\mysection{\@startsection{section}{1}\z@{.7\linespacing\@plus\linespacing}{.5\linespacing}{\normalfont\scshape\centering\S}}
\makeatother

\usepackage{amsmath,amssymb,amsthm}
\usepackage{mathrsfs}
\usepackage{mathabx}\changenotsign
\usepackage{dsfont}
\usepackage{bbm}

\usepackage{xcolor}
\usepackage[backref=section]{hyperref}
\usepackage[ocgcolorlinks]{ocgx2}
\hypersetup{
	colorlinks=true,
	linkcolor={red!60!black},
	citecolor={green!60!black},
	urlcolor={blue!60!black},
}

\usepackage[open,openlevel=2,atend]{bookmark}

\usepackage[abbrev,msc-links,backrefs]{amsrefs}
\usepackage{doi}

\renewcommand{\PrintDOI}[1]{\doi{#1}}

\usepackage[T1]{fontenc}
\usepackage{lmodern}
\usepackage[babel]{microtype}
\usepackage[english]{babel}

\usepackage{geometry}
\numberwithin{equation}{section}
\numberwithin{figure}{section}

\usepackage{enumitem}

\def\alabel{\upshape({\itshape \alph*\,})}

\let\polishlcross=\l
\def\l{\ifmmode\ell\else\polishlcross\fi}

\let\setminus=\smallsetminus

\makeatletter
\def\moverlay{\mathpalette\mov@rlay}
\def\mov@rlay#1#2{\leavevmode\vtop{   \baselineskip\z@skip \lineskiplimit-\maxdimen
		\ialign{\hfil$\m@th#1##$\hfil\cr#2\crcr}}}
\newcommand{\charfusion}[3][\mathord]{
	#1{\ifx#1\mathop\vphantom{#2}\fi
		\mathpalette\mov@rlay{#2\cr#3}
	}
	\ifx#1\mathop\expandafter\displaylimits\fi}
\makeatother

\newcommand{\dcup}{\charfusion[\mathbin]{\cup}{\cdot}}

\DeclareFontFamily{U}  {MnSymbolC}{}
\DeclareSymbolFont{MnSyC}         {U}  {MnSymbolC}{m}{n}
\DeclareFontShape{U}{MnSymbolC}{m}{n}{
	<-6>  MnSymbolC5
	<6-7>  MnSymbolC6
	<7-8>  MnSymbolC7
	<8-9>  MnSymbolC8
	<9-10> MnSymbolC9
	<10-12> MnSymbolC10
	<12->   MnSymbolC12}{}
\DeclareMathSymbol{\powerset}{\mathord}{MnSyC}{180}
\DeclareMathSymbol{\YY}{\mathord}{MnSyC}{42}

\usepackage{tikz}
\usetikzlibrary{calc,decorations.pathmorphing}
\usetikzlibrary{arrows,decorations.pathreplacing}
\pgfdeclarelayer{background}
\pgfdeclarelayer{foreground}
\pgfdeclarelayer{front}
\pgfsetlayers{background,main,foreground,front}
\definecolor{uuuuuu}{rgb}{0.27,0.27,0.27}
\definecolor{sqsqsq}{rgb}{0.1255,0.1255,0.1255}

\usepackage{multicol}
\usepackage{subcaption}
\let\epsilon=\varepsilon
\let\eps=\epsilon
\let\phi=\varphi
\let\rho=\varrho
\let\theta=\vartheta

\def\EE{{\mathds E}}
\def\NN{{\mathds N}}

\def\RR{{\mathds R}}

\def\cupp{{\mathrm{cup}}}
\def\capp{{\mathrm{cap}}}

\newcommand{\cF}{\mathcal{F}}

\newtheoremstyle{note}  {4pt}  {4pt}  {\sl}  {}  {\bfseries}  {.}  {.5em}          {}
\newtheoremstyle{introthms}  {3pt}  {3pt}  {\itshape}  {}  {\bfseries}  {.}  {.5em}          {\thmnote{#3}}
\newtheoremstyle{remark}  {2pt}  {2pt}  {\rm}  {}  {\bfseries}  {.}  {.3em}          {}

\theoremstyle{plain}
\newtheorem{theorem}{Theorem}[section]
\newtheorem{lemma}[theorem]{Lemma}
\newtheorem{prop}[theorem]{Proposition}
\newtheorem{corollary}[theorem]{Corollary}

\newtheorem{proposition}[theorem]{Proposition}
\newtheorem{problem}[theorem]{Problem}
\newtheorem{constr}[theorem]{Construction}
\newtheorem{conjecture}[theorem]{Conjecture}
\newtheorem{cor}[theorem]{Corollary}
\newtheorem{fact}[theorem]{Fact}
\newtheorem{claim}[theorem]{Claim}

\theoremstyle{note}

\newtheorem{definition}[theorem]{Definition}

\theoremstyle{remark}

\usepackage{lineno}
\newcommand*\patchAmsMathEnvironmentForLineno[1]{
	\expandafter\let\csname old#1\expandafter\endcsname\csname #1\endcsname
	\expandafter\let\csname oldend#1\expandafter\endcsname\csname end#1\endcsname
	\renewenvironment{#1}
	{\linenomath\csname old#1\endcsname}
	{\csname oldend#1\endcsname\endlinenomath}}
\newcommand*\patchBothAmsMathEnvironmentsForLineno[1]{
	\patchAmsMathEnvironmentForLineno{#1}
	\patchAmsMathEnvironmentForLineno{#1*}}
\AtBeginDocument{
	\patchBothAmsMathEnvironmentsForLineno{equation}
	\patchBothAmsMathEnvironmentsForLineno{align}
	\patchBothAmsMathEnvironmentsForLineno{flalign}
	\patchBothAmsMathEnvironmentsForLineno{alignat}
	\patchBothAmsMathEnvironmentsForLineno{gather}
	\patchBothAmsMathEnvironmentsForLineno{multline}
}

\usepackage{scalerel}

\makeatletter
\newcommand{\overrighharpoonup}[1]{\ThisStyle{%
		\vbox {\m@th\ialign{##\crcr
				\rightharpoonupfill \crcr
				\noalign{\kern-\p@\nointerlineskip}
				$\hfil\SavedStyle#1\hfil$\crcr}}}}

\def\rightharpoonupfill{%
	$\SavedStyle\m@th\mkern+0.8mu\cleaders\hbox{$\shortbar\mkern-4mu$}\hfill\rightharpoonuptip\mkern+0.8mu$}

\def\rightharpoonuptip{%
	\raisebox{\z@}[2pt][1pt]{\scalebox{0.55}{$\SavedStyle\rightharpoonup$}}}

\def\shortbar{%
	\smash{\scalebox{0.55}{$\SavedStyle\relbar$}}}
\makeatother

\let\lra=\longrightarrow

\makeatletter
\newsavebox\myboxA
\newsavebox\myboxB
\newlength\mylenA

\newcommand*\xoverline[2][0.75]{%
	\sbox{\myboxA}{$\m@th#2$}%
	\setbox\myboxB\null% Phantom box
	\ht\myboxB=\ht\myboxA%
	\dp\myboxB=\dp\myboxA%
	\wd\myboxB=#1\wd\myboxA% Scale phantom
	\sbox\myboxB{$\m@th\overline{\copy\myboxB}$}%  Overlined phantom
	\setlength\mylenA{\the\wd\myboxA}%   calc width diff
	\addtolength\mylenA{-\the\wd\myboxB}%
	\ifdim\wd\myboxB<\wd\myboxA%
	\rlap{\hskip 0.5\mylenA\usebox\myboxB}{\usebox\myboxA}%
	\else
	\hskip -0.5\mylenA\rlap{\usebox\myboxA}{\hskip 0.5\mylenA\usebox\myboxB}%
	\fi}
\makeatother

\DeclareSymbolFont{symbolsC}{U}{txsyc}{m}{n}
\SetSymbolFont{symbolsC}{bold}{U}{txsyc}{bx}{n}
\DeclareFontSubstitution{U}{txsyc}{m}{n}
\DeclareMathSymbol{\strictif}{\mathrel}{symbolsC}{74}

\DeclareSymbolFont{stmry}{U}{stmry}{m}{n}
\DeclareMathSymbol\arrownot\mathrel{stmry}{"58}
\DeclareMathSymbol\Arrownot\mathrel{stmry}{"59}

\let\sm=\smallsetminus
\begin{document}
\title[The feasible region of induced graphs]{The feasible region of induced graphs}

\author{Xizhi Liu}
\address{Department of Mathematics, Statistics, and Computer Science, University of Illinois,
Chicago, IL 60607 USA}
\email{xliu246@uic.edu}
\email{mubayi@uic.edu}
\thanks{The first and second author's research is partially supported by NSF awards 
DMS-1763317 and DMS-1952767.}

\author{Dhruv Mubayi}
%\address{Department of Mathematics, Statistics, and Computer Science, University of Illinois,
%Chicago, IL 60607 USA}
%\email{mubayi@uic.edu}
%

\author{Christian Reiher}
\address{Fachbereich Mathematik, Universit\"at Hamburg, Hamburg, Germany}
\email{Christian.Reiher@uni-hamburg.de}

\subjclass[2010]{}

%%%%%%%%%%%%%%%%%%%%%%%%%%%%%%%%%%%%%%%%%%%%%%%%%
\begin{abstract}
The feasible region $\Omega_{{\rm ind}}(F)$ of a graph $F$ is the collection of 
points $(x,y)$ in the unit square such that there exists a sequence of graphs 
whose edge densities approach~$x$ and whose induced $F$-densities approach~$y$. 
A complete description of $\Omega_{{\rm ind}}(F)$ is not known for any~$F$ with 
at least four vertices that is not a clique or an independent set.
The feasible region provides a lot of combinatorial information about~$F$.
For example, the supremum of~$y$ over all $(x,y)\in \Omega_{{\rm ind}}(F)$ is 
the inducibility of $F$ and $\Omega_{{\rm ind}}(K_r)$ yields the Kruskal-Katona 
and clique density theorems.

We begin a systematic study of $\Omega_{{\rm ind}}(F)$ by proving some general 
statements about the shape of $\Omega_{{\rm ind}}(F)$ and giving results for some 
specific graphs $F$. Many of our theorems apply to the more general setting of 
quantum graphs. For example, we prove a  bound for quantum graphs that generalizes 
an old result of Bollob\'as for the number of cliques in a graph with given edge density.  
We also consider the problems of determining $\Omega_{{\rm ind}}(F)$ when $F=K_r^-$, $F$ 
is a star, or $F$ is a complete bipartite graph. In the case of $K_r^-$ our results sharpen 
those predicted by the  edge-statistics conjecture of Alon et.\ al.\ while also extending 
a theorem of Hirst for $K_4^-$ that was proved using computer aided techniques and 
flag algebras. The case of the 4-cycle seems particularly 
interesting and 
we conjecture that $\Omega_{{\rm ind}}(C_4)$ is determined by the solution to the triangle 
density problem, which has been solved by Razborov. 
\end{abstract}
%%%%%%%%%%%%%%%%%%%%%%%%%%%%%%%%%%%%%%%%%%%%%%%%%
\maketitle
\section{Introduction}
\subsection{Feasible regions}
Given a graph $G$ denote by $V(G)$ and $E(G)$ the vertex set and edge set of $G$ respectively.
Let $v(G) = |V(G)|$, $e(G) = |E(G)|$, and call $\rho(G) = e(G)/\binom{v(G)}{2}$ the 
{\em edge density} of $G$.
For two graphs $F$ and $G$ denote by $N(F,G)$ the number of induced copies of $F$ in $G$,
and let $\rho(F,G) = {N(F,G)}/{\binom{v(G)}{v(F)}}$ be the {\em induced $F$-density} of $G$.

A {\em quantum graph} $Q$ is a formal linear combination of finitely many graphs,
i.e., an expression of the form  
\[
	Q = \sum_{i=1}^{m}\lambda_i F_i\,,
\]
where $m$ is a nonnegative integer, the numbers $\lambda_1, \dots, \lambda_m$ are real, 
and $F_1, \dots, F_m$
are graphs. We call $F_i$ a {\em constituent} of $Q$ if $\lambda_i \neq 0$. 
Two quantum graphs $Q$, $Q'$ are equal if they have the same constituents and the same 
(nonzero) coefficients for each constituent.
The {\em complement} of~$Q$ is $\overline{Q} = \sum_{i=1}^{m}\lambda_i \overline{F}_i$, 
where $\overline{F}_i$ denotes the complement of $F_i$ for each $i\in [m]$. A quantum 
graph $Q$ is {\em self-complementary} if $Q = \overline{Q}$.
Every graph parameter~$f$ can be extended linearly to quantum graphs by stipulating  
$f(Q) = \sum_{i=1}^{m}\lambda_i f(F_i)$.
In particular,
\begin{align}
N(Q,G) = \sum_{i=1}^{m}\lambda_i N(F_i, G) \quad{\rm and}\quad
\rho(Q,G) = \sum_{i=1}^{m}\lambda_i \rho(F_i, G)\,. \notag
\end{align}
The main notion investigated in this article is the following.

\begin{definition}[Feasible region]\label{DFN:induced-feasible-region}
Let $Q = \sum_{i=1}^{m}\lambda_i F_i$ be a quantum graph.
\begin{itemize}
\item
A sequence $\left(G_{n}\right)_{n=1}^{\infty}$ of graphs is \emph{$Q$-good} if
$\lim_{n\to \infty}v(G_n) = \infty$, $\lim_{n\to\infty}\rho(G_n)$ exists, and 
for every $i\in[m]$ the limit $\lim_{n\to\infty}\rho(F_i,G_n)$ exists.
\item
A $Q$-good sequence of graphs $\left(G_{n}\right)_{n=1}^{\infty}$ \emph{realizes} 
a point $(x,y)\in [0,1]\times \RR$ if
\begin{align}
\lim_{n\to\infty}\rho(G_n) = x \quad{\rm and}\quad \lim_{n\to\infty}\rho(Q,G_n) = y. \notag
\end{align}
\item
The \emph{feasible region} $\Omega_{{\rm ind}}(Q)$ of (induced) $Q$ is the collection of 
points $(x,y) \in [0,1] \times \RR$ realized by some $Q$-good 
sequence $\left(G_{n}\right)_{n=1}^{\infty}$.
\end{itemize}
\end{definition}

We commence a systematic study of the feasible region of quantum graphs $Q$.
As we shall see soon, $\Omega_{{\rm ind}}(Q)$ is determined by its boundary, 
so it suffices to consider for every $x\in [0,1]$ the numbers
\[
	i(Q,x) = \inf\{y\colon (x,y)\in \Omega_{{\rm ind}}(Q)\} 
	\qquad{\rm and}\qquad
	I(Q,x) = \sup\{y\colon (x,y)\in \Omega_{{\rm ind}}(Q)\}\,.
\]

Determining the values of $i(Q,x)$ and $I(Q,x)$ under some constraints is a central topic in extremal combinatorics.
For example, the classical Kruskal-Katona theorem~\cites{Kru63,Ka68} implies 
\[
	I(K_{r},x) = x^{r/2} 
	\quad \text{for all $r\ge 2$ and $x\in[0,1]$}\,.
\]
Tur\'{a}n's seminal theorem~\cite{TU41} and supersaturation show that for 
every integer $r \ge 3$,
\[
	i(K_{r},x)> 0 \qquad \Longleftrightarrow \qquad x > (r-2)/(r-1)\,.
\]
Determining $i(K_{r},x)$ for all $x > (r-2)/(r-1)$ is highly nontrivial
and was solved for~$r = 3$ by Razborov~\cite{Ra08}, for $r = 4$ by Nikiforov~\cite{Niki11}, 
and for all $r$ by the third author~\cite{Re16}.

Regarding quantum graphs with at least two constituents,
a classical result of Goodman~\cite{Go59} says that $i(K_3+\overline{K}_3,x) \ge 1/4$ 
and equality holds only for $x = 1/2$.
Erd\H{o}s~\cite{Er62} conjectured
that $i(K_{r}+\overline{K}_{r},x) \ge 2^{1-\binom{r}{2}}$ for $r\ge 4$ with equality 
for $x = 1/2$. This conjecture was disproved by Thomason~\cite{Thom89} 
for all $r\ge 4$, but even for $r=4$ the minimum value of $i(K_{r}+\overline{K}_{r},x)$ 
is still unknown. 

For a single graph $F$ the function $I(F,x)$ is closely related to 
the {\em inducibility} 
\[
	\mathrm{ind}(F) = \lim_{n\to \infty}\max\left\{\rho(F,G)\colon v(G) = n\right\}
\]
of $F$ introduced by Pippenger and Golumbic~\cite{PG75}. In fact, 
$\mathrm{ind}(F)=\max\{I(F, x)\colon x\in [0, 1]\}$, where the maximum exists due to 
the continuity of $I(F, x)$ (see Theorem~\ref{thm:1347} below). 

Determining the feasible region $\Omega_{{\rm ind}}(F)$ of a single graph $F$ is a special 
case of the more general problem to determine the {\em graph profile} $T(\cF)$ of a given
finite family of graphs $\cF = \{F_1,\dots, F_k\}$.
Here $T(\cF) \subseteq [0,1]^{k}$ is the collection of limit points of
$\left((\rho(F_1,G_i),\ldots, \rho(F_k,G_i))\right)_{i=1}^{\infty}$ with $v(G_i) \to \infty$.
Besides the clique density theorem, very few results are known about graph profiles 
(see~\cites{HN11,HLNPS14,BL16,HN19}).

Our results are of two flavors.
\begin{itemize}
\item
We prove some general results about the shape of $\Omega_{{\rm ind}}(Q)$.
Our main result here is Theorem~\ref{thm:1347},
which states that  $I(Q,x)$ and $i(Q,x)$ are continuous and almost everywhere differentiable.
\item
We study $\Omega_{{\rm ind}}(Q)$ for some specific choices of $Q$ for which $\mathrm{ind}(Q)$ 
has been investigated by many researchers.
We focus on quantum graphs whose constituents are complete multipartite graphs and 
prove a general upper bound for $I(Q,x)$.
Prior to this work, $\Omega_{{\rm ind}}(F)$ for a single graph $F$ was determined only 
when $F$ is a  clique or an independent set. Here we extend this to the case $F=K_{1,2}$ 
and also obtain results for complete bipartite graphs.
Furthermore we study $\Omega_{{\rm ind}}(K_r^-)$, where $K_r^-$ arises from the clique $K_r$ 
by the deletion of a single edge. As a consequence of our results, we determine the 
inducibility $\mathrm{ind}(K_r^-)$, which is new for $r\ge 5$. 
\end{itemize}

\subsection{General results}\label{SUBSEC:general-results-intro}

The following result describes the shape of the feasible region of an arbitrary 
quantum graph. 
 
\begin{theorem}\label{thm:1347}
	For every quantum graph $Q$ we have 
	\[
		\Omega_{{\rm ind}}(Q)
		=
		\bigl\{(x, y)\in [0, 1]\times \RR\colon i(Q, x)\le y\le I(Q, x)\bigr\}\,.
	\]
	Moreover, the boundary functions $i(Q, x)$ and $I(Q, x)$ are continuous and 
	almost everywhere differentiable. 
\end{theorem}

In contrast to Theorem~\ref{thm:1347} Hatami and Norin~\cite{HN19}
gave an example of a finite family $\cF$ of graphs such that the intersection of
the graph profile $T(\cF)$ with some hyperplane has a nowhere differentiable boundary.

For every quantum graph $Q$ the feasible regions of $Q$, $-Q$ and $\overline{Q}$ 
are closely related. Indeed, using the formulae  
\begin{align}%\label{equ:induciblity-complement}
N(F,G) = N(\overline{F},\overline{G})
\quad{\rm and}\quad
\rho(F,G) = \rho(\overline{F},\overline{G}), \notag
\end{align} 
which are valid for all graphs $F$ and $G$, one easily confirms the following 
observation. 

\vbox{
\begin{fact}\label{FACT:complement-minus-inducibility}
Let $Q$ be a quantum graph. 
\begin{enumerate}[label=\alabel]
\item\label{it:13a} The feasible regions of $Q$ and $-Q$ are symmetric to each other about 
the $x$-axis. Hence, $I(-Q, x)=-i(Q, x)$ and $i(-Q, x)=-I(Q, x)$ hold for all $x\in [0, 1]$.
\item\label{it:13b} The feasible regions of $Q$ and $\overline{Q}$ are symmetric to 
each other about the line ${x=1/2}$. 
Thus we have $I(Q,x) = I(\overline{Q},1-x)$ and $i(Q,x) = i(\overline{Q},1-x)$
for every $x\in [0, 1]$.
In particular, if $Q$ is self-complementary, then $I(Q,x) = I(Q,1-x)$ and $i(Q,x) = i(Q,1-x)$,
i.e. the functions $I(Q,x)$ and $i(Q,x)$ are symmetric around $x=1/2$. \qed
\end{enumerate}
\end{fact}
}

The next result shows that for most single graphs $F$ the lower boundary 
function $i(F, x)$ vanishes identically. The only exceptions occur when $F$
is a clique or the complement of a clique, in which case $i(F, x)$ is given by 
the clique density theorem (see Theorem~\ref{thm:cdt})
and Fact~\ref{FACT:complement-minus-inducibility}\ref{it:13b}.

\begin{proposition}\label{PROP:lower-bound-is-zero}
If $F$ denotes a graph which is neither complete nor empty, 
then ${i(F,x) = 0}$ for all $x\in [0,1]$.
\end{proposition}

We proceed with some estimates based on random graphs. 
Given a quantum graph $Q = \sum_{i=1}^{m}\lambda_i F_i$ we define
\begin{align}%\label{equ:pseudorandom-bound}
{\rm rand}(Q,x)
= \sum_{i\in[m]}\lambda_{i} 
\frac{\left(v(F_i)\right)!}{|{\rm Aut}(F_i)|}
x^{e(F_i)}(1-x)^{e(\overline{F}_i)}
\quad \text{ for every } x\in [0, 1]\,, \notag
\end{align}
where ${\rm Aut}(F_i)$ is the automorphism group of $F_i$ for $i\in [m]$. 
Equivalently, 
\[
	{\rm rand}(Q,x)
	=
	\lim_{n\to\infty}\EE \,\rho(Q, G(n, x))\,,
\]
where $G(n, x)$ denotes the standard binomial random graph. It is well known
that the random variables $\rho(G(n, x))$, $\rho(Q, G(n, x))$ are tightly 
concentrated around their expectations. This shows the following observation.

\begin{fact}\label{PROP:general-upper-lower-bound-for-quantum-graphs}
If $Q$ denotes a quantum graph and $x\in [0,1]$, then
\begin{align}
I(Q,x)  \ge {\rm rand}(Q,x) \ge i(Q,x)\,. \notag
\end{align}
In particular, for a single graph $F$ the inequality $I(F,x) > 0$ holds 
for all $x\in(0,1)$. \qed
\end{fact}

Let $P_{4,1}$ be the $5$-vertex graph that is the disjoint union of a path on $4$ vertices 
and an isolated vertex. It was asked in~\cite{EL15} whether the inducibility of some graph 
is achieved by a random graph and, in particular, whether 
the inducibility ${\mathrm{ind}}(P_{4,1})$ is 
achieved by the Erd\H{o}s-R\'{e}nyi random graph $G(n, 3/10)$. Here we pose an easier 
question of a similar flavor.

\begin{problem}\label{PROB:is-pseudorandom-bound-tight}
Do there exist a graph $F$ and some $x\in(0,1)$ such that $I(F,x) = {\rm rand}(F,x)$?
\end{problem}

%%%%%%%%%%%%%%%%%%%%%%%%%%%%%%%%%%
\subsection{Complete multipartite graphs}
We now present our results on $I(Q, x)$ for specific quantum graphs $Q$.
Our focus is on quantum graphs whose constituents are complete multipartite graphs
(a graph whose edge set is empty is viewed as complete multipartite with only one part).
A case of particular interest is  $Q = K_r+\overline{K}_r$ for $r\ge 3$.
Goodman \cite{Go59} proved that for every graph $G$ on $n$ vertices
%\begin{align}%\label{equ:Goodman-K3-bar-K3}
$\rho(K_3+\overline{K}_3,G) \ge 1/4+o(1)$
%, \notag\end{align}
and the random graph $G(n,1/2)$ shows that this bound is tight.
Therefore, $i(K_3+\overline{K}_3,x) \ge 1/4$ and equality holds when $x = 1/2$.
Combining Goodman's result~\cite{Go59} with a theorem of Olpp~\cite{Olpp96}
one can determine $\Omega_{{\rm ind}}(K_3+\overline{K}_3)$ completely.

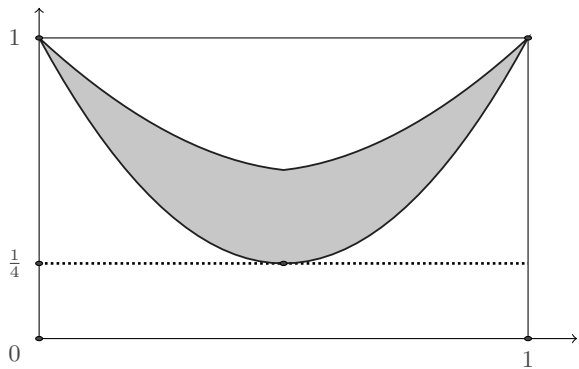
\begin{figure}[htbp]
\centering
\begin{tikzpicture}[xscale=6.5,yscale=4]
\draw [->] (0,0)--(1.1,0);
\draw [->] (0,0)--(0,1.1);
\draw (0,1)--(1,1);
\draw (1,0)--(1,1);
\draw [line width=1pt,dash pattern=on 1pt off 1.2pt,domain=0:1] plot(\x,{1/4});
\draw[line width=0.7pt,color=sqsqsq,fill=sqsqsq,fill opacity=0.25]
(0, 1)--(0.025, 0.963206)--(0.05, 0.927836)--(0.075, 0.89391)--(0.1, 0.861445)--
(0.125, 0.830463)--(0.15, 0.800984)--(0.175,0.77303)--(0.2, 0.746625)--(0.225, 0.721792)--
(0.25, 0.698557)--(0.275, 0.676946)--(0.3, 0.656986)--(0.325, 0.638707)--(0.35,0.62214)--
(0.375, 0.607318)--(0.4, 0.594274)--(0.425, 0.583046)--(0.45, 0.573673)--
(0.475, 0.566196)--(0.5, 0.56066)--(0.525,0.566196)--(0.55, 0.573673)--
(0.575, 0.583046)--(0.6, 0.594274)--(0.625, 0.607318)--(0.65, 0.62214)--
(0.675, 0.638707)--(0.7,0.656986)--(0.725, 0.676946)--(0.75, 0.698557)--
(0.775, 0.721792)--(0.8, 0.746625)--(0.825, 0.77303)--(0.85, 0.800984)--
(0.875,0.830463)--(0.9, 0.861445)--(0.925, 0.89391)--(0.95, 0.927836)--
(0.975, 0.963206)--(1, 1)--
(0.975, 0.926875)--(0.95, 0.8575)--(0.925, 0.791875)--(0.9,0.73)--
(0.875, 0.671875)--(0.85, 0.6175)--(0.825, 0.566875)--(0.8, 0.52)--(0.775, 0.476875)--
(0.75, 0.4375)--(0.725, 0.401875)--(0.7,0.37)--(0.675, 0.341875)--(0.65, 0.3175)--
(0.625, 0.296875)--(0.6,0.28)--(0.575, 0.266875)--(0.55, 0.2575)--(0.525, 0.251875)--
(0.5,0.25)--(0.475, 0.251875)--(0.45, 0.2575)--(0.425, 0.266875)--(0.4,0.28)--
(0.375, 0.296875)--(0.35, 0.3175)--(0.325, 0.341875)--(0.3,0.37)--(0.275, 0.401875)--
(0.25, 0.4375)--(0.225, 0.476875)--(0.2,0.52)--(0.175, 0.566875)--(0.15, 0.6175)--
(0.125, 0.671875)--(0.1,0.73)--(0.075, 0.791875)--(0.05, 0.8575)--(0.025, 0.926875)--(0, 1);

\begin{scriptsize}
\draw [fill=uuuuuu] (1,0) circle (0.2pt);
\draw[color=uuuuuu] (1,0-0.07) node {$1$};
\draw [fill=uuuuuu] (0,0) circle (0.2pt);
\draw[color=uuuuuu] (0-0.05,0-0.05) node {$0$};
\draw [fill=uuuuuu] (0,1) circle (0.2pt);
\draw[color=uuuuuu] (0-0.05,1) node {$1$};
\draw [fill=uuuuuu] (1,1) circle (0.2pt);
\draw [fill=uuuuuu] (0,1/4) circle (0.2pt);
\draw[color=uuuuuu] (0-0.05,1/4) node {$\frac{1}{4}$};
\draw [fill=uuuuuu] (1/2,1/4) circle (0.2pt);
\end{scriptsize}
\end{tikzpicture}
\caption{$\Omega_{{\rm ind}}(K_3+\overline{K}_3)$ is the shaded area above.}
\label{fig:feasible-region-K3+k3-bar}
\end{figure}

\begin{theorem}[Goodman~\cite{Go59}, Olpp~\cite{Olpp96}]\label{THM:feasible-region-K3-K3-bar}
For every $x\in [0,1]$ we have
\begin{align*}
\pushQED{\qed} 
i(K_3+\overline{K}_3,x) & = 1-3x+3x^2 \quad{\rm and}\\
I(K_3+\overline{K}_3,x) & = 1-3\min\left\{x-x^{3/2}, (1-x)-(1-x)^{3/2}\right\}\,. \qedhere 
\popQED
\end{align*}
\end{theorem}

For $r\ge 4$ determining $\Omega_{{\rm ind}}(K_r+\overline{K}_r)$ seems  beyond current methods.

\begin{problem}\label{PROB:feasible-region-cliques-independent-set}
Determine $\Omega_{{\rm ind}}(K_r+\overline{K}_r)$ for $r\ge 4$.
\end{problem}

Another well-studied problem concerns the determination of $\Omega_{{\rm ind}}(K_r)$
for $r\ge 3$. We already mentioned that $I(K_r, x)=x^{r/2}$ follows from the 
Kruskal-Katona theorem~\cites{Kru63,Ka68}. For the lower bound $i(K_r, x)$ we consider 
(independently of $r$) the following complete multipartite graphs. 

\begin{constr}\label{con:cdt}
For integers $n \ge k \ge 2$ and real $x\in\bigl(\frac{k-2}{k-1}, \frac{k-1}{k}\bigr]$
let $H^\star(n,x)$ be the complete $k$-partite graph on $n$ vertices with parts 
$V_1,\ldots,V_k$ of sizes ${|V_1| = \cdots = |V_{k-1}| = \lfloor \alpha_k n \rfloor}$ 
and $|V_k| = n-(k-1)\lfloor \alpha_k n \rfloor$,
where
\begin{align}
\alpha_k = \frac{1}{k}\left(1+\sqrt{1-\frac{k}{k-1}x}\right). \notag
\end{align}
Moreover, $H^\star(n, 0)$ and $H^\star(n , 1)$ denote the empty and the complete graph 
on $n$ vertices. 
\end{constr}

One checks immediately that $\lim_{n\to\infty} \rho(H^\star(n,x))=x$ holds for 
every $x\in [0, 1]$. Consequently, for every $r\ge 2$ the function 
$g_r(x)=\lim_{n\to\infty} \rho(K_r, H^\star(n,x))$ is an upper bound on $i(K_r, x)$.

A more explicit description of $g_r$ is as follows. Clearly $g_r(x)=0$ holds for 
every $x\le \frac{r-2}{r-1}$ and $g(1)=1$. If $x \in \bigl(\frac{r-2}{r-1}, 1\bigr)$ 
there exists a unique integer $k\ge r$ 
such that $x\in \bigl(\frac{k-2}{k-1}, \frac{k-1}{k}\bigr]$ and a short calculation 
reveals 
\begin{align}%\label{equ:clique-density}
g_{r}(x)
= \frac{(k)_{r}}{k^r} \left(1+\sqrt{1-\frac{k}{k-1}x}\right)^{r-1}\left(1-(r-1)\sqrt{1-\frac{k}{k-1}x}\right), \notag
\end{align}
where $(k)_{r} = k(k-1)\cdots (k-r+1)$. Lov\'asz and Simonovits conjectured in the seventies 
that this function coincides with $i(K_r,x)$ and the third author proved that this is indeed 
the case. 

\begin{theorem}[Clique density theorem, Reiher~\cite{Re16}]~\label{thm:cdt}
For all integers $r \ge 3$ and real $x\in[0, 1]$ we have $i(K_r,x) = g_{r}(x)$. \qed
\end{theorem}

The non-asymptotic problem to determine for given natural numbers~$n$ and~$m$ 
the exact minimum number of $r$-cliques an $n$-vertex graph with $m$ edges needs to contain is 
still wide open in general. But for triangles there has recently been spectacular progress by 
Liu, Pikhurko, and Staden~\cite{LPS}.

Easy calculations show that the function $g_r(x)$ is non-differentiable at the 
critical values $x=1-1/q$, where $q\ge r-1$ denotes an integer. Moreover, $g_r(x)$
is piecewise concave between any two consecutive critical values. 
An old result of Bollob\'{a}s~\cite{Bo76} (proved long before the clique density theorem)
asserts that the piece-wise linear function interpolating between the critical values 
of $g_r(x)$ is a lower bound on $i(K_r, x)$. Here we extend this result to quantum graphs 
whose constituents are complete multipartite graphs.

To state this generalization we need the following concepts.  
For every positive integer $r\ge 2$ and every quantum graph~$Q$ we 
define the {\em complete $r$-partite feasible region} $\Omega_{{\rm ind-r}}(Q)$  
to be the collection of all points in $[0,(r-1)/r]\times \RR$ that can be realized 
by a $Q$-good sequence $\left(G_{n}\right)_{n=1}^{\infty}$
of complete $r$-partite graphs (isolated vertices are not allowed).
For  $x\in [0,(r-1)/r]$, let
\[
	i_{r}(Q,x) = \inf\{y\colon (x,y)\in \Omega_{{\rm ind-r}}(Q)\} 
	\quad{\rm and} \quad
	I_{r}(Q,x) = \sup\{y\colon (x,y)\in \Omega_{{\rm ind-r}}(Q)\}
	\,.
\]
Optimizing over $r$ we put 
\[
	m(Q,x) = \inf\left\{i_{r}(Q,x)\colon r\ge \left\lceil \frac{1}{1-x} \right\rceil\right\}
	\quad{\rm and} \quad
	M(Q,x)  = \sup\left\{I_{r}(Q,x)\colon r\ge \left\lceil \frac{1}{1-x} \right\rceil\right\} 	
\]
for every quantum graph $Q$ and every real $x\in [0, 1)$ as well as 
\[
	m(Q, 1)=M(Q, 1)=\lim_{n\to\infty}\rho(Q, K_n)\,.
\]
Clearly, we have 
\[
	i(Q,x) \le m(Q,x) \le M(Q,x) \le I(Q,x)\,.
\]

Next we observe that for every bounded function $f\colon [0,1]\longrightarrow \RR$ there exist 
a point-wise minimum concave function $\capp(f)\ge f$ and, similarly, a maximum convex
function $\cupp(f)\le f$. In fact, $\capp(f)$ is given by  
\[
	\capp(f)(x)=\sup\Bigl\{\lambda_1f(x_1)+\dots+\lambda_nf(x_n)\colon 
		n\ge 1,
		(\lambda_1, \dots, \lambda_n)\in\Delta_{n-1}, 
		\text{ and }
		\sum_{i=1}^n \lambda_i x_i=x \Bigr\}
\]
for all $x\in [0, 1]$, where 
\[
	\Delta_{n-1}
	=
	\bigl\{(\lambda_1, \dots, \lambda_n)\in [0, 1]^n\colon 
	\lambda_1+\dots+\lambda_n=1\bigr\}
\]
denotes the $(n-1)$-dimensional standard simplex. Moreover, replacing the supremum 
by an infimum one obtains a formula for $\cupp(f)(x)$. 

\begin{theorem}\label{THM:inducibility-multi-partite-all}
Let $Q = \sum_{i=1}^m\lambda_i F_i$ be a quantum graph all of whose constituents 
are complete multipartite graphs. 
\begin{enumerate}[label=\alabel]
	\item\label{it:110a} If every $F_i$ with $\lambda_i > 0$ is complete, then 
		\begin{align*}%\label{equ:multi-partite-all}
				i(Q,x) \ge \cupp\left(m(Q,\cdot)\right)(x)
				\quad \text{ for all } x\in[0,1]\,. 
		\end{align*}
	\item\label{it:110b} If every $F_i$ with $\lambda_i < 0$ is complete, then 
		\begin{align*}%\label{equ:multi-partite-all}
				I(Q,x) \le \capp\left(M(Q,\cdot)\right)(x) 
				\quad \text{ for all } x\in[0,1]\,. 
		\end{align*}
\end{enumerate}
\end{theorem}

The aforementioned result of Bollob\'{a}s is the case $Q=K_r$ of 
Theorem~\ref{THM:inducibility-multi-partite-all}\ref{it:110a}.

\subsection{Almost complete graphs}
For every integer $t\ge 3$ we let $K_t^-$ denote the graph obtained from a clique $K_t$
by deleting one edge. As these graphs are neither complete nor empty, 
Proposition~\ref{PROP:lower-bound-is-zero} tells us that the feasible 
regions $\Omega_{\mathrm{ind}}(K_t^-)$ are completely determined by the 
functions $I(K_t^-, x)$. 
For $t=3$ we have the following exact result showing that the graphs $H^\star(n, x)$ minimizing 
the triangle density also maximize the induced $K_3^-$-density. 

\begin{theorem}\label{THM:inducibility-S2}
The equality $I(K_3^-,x) = \frac{3}{2}\left(x - g_3(x)\right)$ holds for all $x\in[0,1]$.
\end{theorem}

\begin{figure}[htbp]
\centering
\begin{tikzpicture}[xscale=6.5,yscale=5]
\draw [->] (0,0)--(1.1,0);
\draw [->] (0,0)--(0,0.85);

\draw [line width=0.5pt,dash pattern=on 1pt off 1.2pt,domain=0:1] plot(\x,{3/4});
\draw [line width=0.5pt,dash pattern=on 1pt off 1.2pt] (1/2,0)--(1/2,3/4);
\draw [line width=0.5pt,dash pattern=on 1pt off 1.2pt] (2/3,0)--(2/3,2/3);
\draw [line width=0.5pt,dash pattern=on 1pt off 1.2pt] (3/4,0)--(3/4,9/16);
\draw [line width=0.5pt,dash pattern=on 1pt off 1.2pt] (4/5,0)--(4/5,12/25);
%\draw [line width=0.5pt,dash pattern=on 1pt off 1.2pt] (5/6,0)--(5/6,5/12);
%\draw [line width=0.5pt,dash pattern=on 1pt off 1.2pt] (6/7,0)--(6/7,18/49);
%\draw [line width=0.5pt,dash pattern=on 1pt off 1.2pt] (7/8,0)--(7/8,21/64);

\draw[line width=0.7pt,color=sqsqsq,fill=sqsqsq,fill opacity=0.25]
(0,0)--(1/2,3/4)
--(0.5,0.75)--(0.51,0.742614)--(0.52,0.735459)--(0.53,0.728545)--(0.54,0.721879)--
(0.55,0.715472)--(0.56,0.709333)--(0.57,0.703476)--(0.58,0.697915)--(0.59,0.692666)--
(0.6,0.687749)--(0.61,0.683188)--(0.62,0.679014)--(0.63,0.675266)--(0.64,0.672)--
(0.65,0.669302)--(2/3,2/3)
--(0.673667,0.656279)--(0.680667,0.646121)--(0.687667,0.636205)--(0.694667,0.626545)--
(0.701667,0.617155)--(0.708667,0.608055)--(0.715667,0.599269)--(0.722667,0.590827)--
(0.729667,0.582772)--(0.736667,0.575167)--(3/4,9/16)
--(0.755,0.553211)--(0.76,0.5441)--(0.765,0.535177)--(0.77,0.526457)--(0.775,0.517955)--
(0.78,0.509692)--(0.785,0.501697)--(0.79,0.494012)--(0.795,0.486712)--(4/5,12/25)
--(5/6,5/12)--(6/7,18/49)--(7/8,21/64)--(8/9,8/27)--(9/10,27/100)--(10/11,30/121)--
(11/12,11/48)--(12/13,36/169)--(13/14,39/196)--(14/15,14/75)
--(15/16,45/256)--(16/17,48/289)--(17/18,17/108)--(18/19,54/361)--(19/20,57/400)--
(20/21,20/147)--(21/22,63/484)--(22/23,66/529)--(23/24,23/192)--(24/25,72/625)--
(25/26,75/676)--(26/27,26/243)--(27/28,81/784)--(28/29,84/841)--(29/30,29/300)--(1,0);

\begin{scriptsize}
\draw [fill=uuuuuu] (1,0) circle (0.2pt);
\draw[color=uuuuuu] (1,0-0.07) node {$1$};

\draw [fill=uuuuuu] (1/2,0) circle (0.2pt);
\draw[color=uuuuuu] (1/2,0-0.07) node {$\frac{1}{2}$};
\draw [fill=uuuuuu] (2/3,0) circle (0.2pt);
\draw[color=uuuuuu] (2/3,0-0.07) node {$\frac{2}{3}$};
\draw [fill=uuuuuu] (3/4,0) circle (0.2pt);
\draw[color=uuuuuu] (3/4,0-0.07) node {$\frac{3}{4}$};
\draw [fill=uuuuuu] (4/5,0) circle (0.2pt);
\draw[color=uuuuuu] (4/5,0-0.07) node {$\frac{4}{5}$};

\draw [fill=uuuuuu] (0,0) circle (0.2pt);
\draw[color=uuuuuu] (0-0.05,0-0.05) node {$0$};

\draw [fill=uuuuuu] (0,3/4) circle (0.2pt);
\draw[color=uuuuuu] (0-0.05,3/4) node {$\frac{3}{4}$};
\draw [fill=uuuuuu] (1/2,3/4) circle (0.2pt);

\draw [fill=uuuuuu] (2/3,2/3) circle (0.2pt);
\draw [fill=uuuuuu] (3/4,9/16) circle (0.2pt);
\draw [fill=uuuuuu] (4/5,12/25) circle (0.2pt);
\draw [fill=uuuuuu] (5/6,5/12) circle (0.2pt);
\draw [fill=uuuuuu] (6/7,18/49) circle (0.2pt);
\draw [fill=uuuuuu] (7/8,21/64) circle (0.2pt);
%\draw [fill=uuuuuu] (8/9,8/27) circle (0.2pt);
%\draw [fill=uuuuuu] (9/10,27/100) circle (0.2pt);
%\draw [fill=uuuuuu] (10/11,30/121) circle (0.2pt);
\end{scriptsize}
\end{tikzpicture}
\caption{$\Omega_{{\rm ind}}(K_3^-)$.}
\label{fig:feasible-region-S2}
\end{figure}

For $t\ge 4$ we show a piecewise linear upper bound on $I(K_t^-, x)$ that yields 
the correct value of the inducibility ${\mathrm{ind}}(K_t^-)$. 
In the statement that follows, we set 
\[
	k(t)=
	\begin{cases}
		\lceil (t+1)(3t-8)/6 \rceil & \text {if } t\ne 5, 8, 11, 14, 17, 20 \cr
		(t-2)(3t+1)/6 & \text {if } t= 5, 8, 11, 14, 17, 20.
	\end{cases} 
\]

\begin{theorem}\label{thm:1643}
	For all $t\ge 4$ and $x\in [0, 1]$ we have $I(K_t^-, x)\le h_t(x)$, where $h_t$
	denotes the piecewise linear function interpolating between $h_t(0)=0$ and  
	\[
		h_t(1-1/r)
		= 
		\binom{t}{2}\frac{(r-1)_{t-2}}{r^{t-1}}
		\quad 
		\text{ for } r\ge k(t)\,.
	\]
	Furthermore, 
	\begin{equation}\label{eq:8648}
		{\mathrm{ind}}(K_t^-)
		=
		\binom{t}{2}\frac{(q(t)-1)_{t-2}}{q(t)^{t-1}},
		\quad 
		\text{ where } q(t) = \lceil (t-2)(3t+1)/6\rceil\,.
	\end{equation}
\end{theorem}   

For instance, for $t=4$ we have $q(4)=5$ and, hence, ${\mathrm{ind}}(K_4^-)=72/125$. 
This was originally proved by Hirst~\cite{Hirst14}, whose computer assisted argument 
is based on the flag algebra method. Moreover, Theorem~\ref{thm:1643} yields the upper 
bound $I(K_4^-, x)\le 3x/4$ for $x\in [0, 3/4]$. For small values of $x$ we have 
the following stronger bound.

\begin{prop}\label{prop:1648}
	If $x\in [0, 1/2]$, then $I(K_4^-, x)\le 3x^2/2$.
\end{prop}

\begin{figure}[htbp]
\centering
\begin{tikzpicture}[xscale=6,yscale=6]
\draw [->] (0,0)--(1.1,0);
\draw [->] (0,0)--(0,72/125+0.1);

\draw [line width=0.5pt,dash pattern=on 1pt off 1.2pt] (1/2,0)--(1/2,3/8);
\draw [line width=0.5pt,dash pattern=on 1pt off 1.2pt] (0,3/8)--(1/2,3/8);
\draw [line width=0.5pt,dash pattern=on 1pt off 1.2pt] (3/4,0)--(3/4,9/16);
\draw [line width=0.5pt,dash pattern=on 1pt off 1.2pt] (4/5,0)--(4/5,72/125);
\draw [line width=0.5pt,dash pattern=on 1pt off 1.2pt] (5/6,0)--(5/6,5/9);
\draw [line width=0.5pt,dash pattern=on 1pt off 1.2pt] (6/7,0)--(6/7,180/343);
\draw [line width=0.5pt,dash pattern=on 1pt off 1.2pt] (0,72/125)--(1,72/125);

\draw[line width=0.7pt,color=sqsqsq,fill=sqsqsq,fill opacity=0.25]
(0,0)--(0. 03,0.00135)--(0.06,0.0054)--(0.09,0.01215)--(0.12,0.0216)--(0.15,0.03375)--
(0.18,0.0486)--(0.21,0.06615)--(0.24,0.0864)--(0.27,0.10935)--(0.3,0.135)--(0.33,0.16335)--
(0.36,0.1944)--(0.39,0.22815)--(0.42,0.2646)--(0.45,0.30375)--(0.48,0.3456)--(0.5,0.375)
--(3/4,9/16)--(4/5,72/125)--(5/6,5/9)--(6/7,180/343)--
(7/ 8,63/128)--(8/9,112/243)--(9/10,54/125)--(10/11,540/1331)--
(11/12,55/144)--(12/13,792/2197)--(13/14,117/343)--(14/15,364/1125)--
(15/16,315/1024)--(16/17,1440/4913)--(17/18,68/243)--(18/19,1836/6859)--(19/20,513/2000)--(1,0);

\begin{scriptsize}
\draw [fill=uuuuuu] (1,0) circle (0.2pt);
\draw[color=uuuuuu] (1,0-0.07) node {$1$};
\draw [fill=uuuuuu] (0,0) circle (0.2pt);
\draw[color=uuuuuu] (0-0.05,0-0.05) node {$0$};
\draw [fill=uuuuuu] (0,72/125) circle (0.2pt);
\draw[color=uuuuuu] (0-0.05,72/125) node {$\frac{72}{125}$};
\draw [fill=uuuuuu] (3/4,0) circle (0.2pt);
\draw[color=uuuuuu] (3/4,0-0.07) node {$\frac{3}{4}$};
\draw [fill=uuuuuu] (3/4,9/16) circle (0.2pt);
\draw [fill=uuuuuu] (4/5,0) circle (0.2pt);
\draw[color=uuuuuu] (4/5,0-0.07) node {$\frac{4}{5}$};
\draw [fill=uuuuuu] (4/5,72/125) circle (0.2pt);
\draw [fill=uuuuuu] (5/6,0) circle (0.2pt);
\draw[color=uuuuuu] (5/6,0-0.07) node {$\frac{5}{6}$};
\draw [fill=uuuuuu] (5/6,5/9) circle (0.2pt);
\draw [fill=uuuuuu] (6/7,0) circle (0.2pt);
\draw [fill=uuuuuu] (6/7,180/343) circle (0.2pt);
\draw[color=uuuuuu] (1/2,0-0.07) node {$\frac{1}{2}$};
\draw [fill=uuuuuu] (1/2,3/8) circle (0.2pt);
\draw [fill=uuuuuu] (0,3/8) circle (0.2pt);
\draw[color=uuuuuu] (0-0.05,3/8) node {$\frac{3}{8}$};
\end{scriptsize}
\end{tikzpicture}
\caption{$\Omega_{{\rm ind}}(K_{4}^{-})$ is contained in the shaded area above.}
\label{fig:feasible-region-K4-minus}
\end{figure}
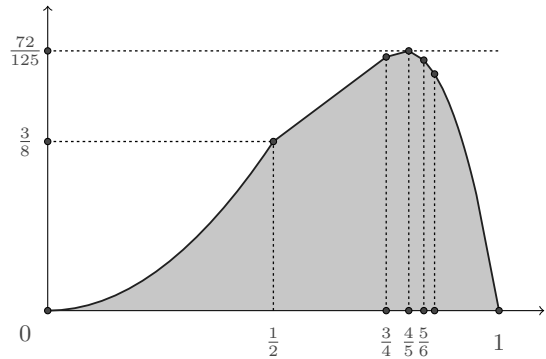

Finally, we remark that our determination of ${\mathrm{ind}}(K_t^-)$ 
in~\eqref{eq:8648} implies
\begin{equation}\label{eq:8751}
	\lim_{t\to\infty} {\mathrm{ind}}(K_t^-)=1/e\,.
\end{equation}

This is closely related to the so-called {\it edge-statistics conjecture} of Alon,  
Hefetz, Krivelevich, and Tyomkyn~\cite{AHKT}. Given positive integers $k$ and 
$\ell<\binom k2$ let the quantum graph $Q_{k, \ell}$ be the sum of all $k$-vertex 
graphs with $\ell$ edges. Alon et al.\ conjectured $\mathrm{ind}(Q_{k, \ell})\le 1/e+o_k(1)$,
where $o_k(1)$ is a quantity tending to zero as $k$ tends to infinity.
They also proved this for some range of $\ell$. Following the work of Kwan, Sudakov, 
and Tran~\cite{KST19}, the edges statistics conjecture was resolved by 
Fox and Sauermann~\cite{FS}
and, independently, by Martinsson, Mousset, Noever, and Truji\'{c}~\cite{MMNT}.
Part of the original motivation for the edges statistics conjecture was the observation 
that for $\ell=1$ we have $Q_{k, 1}=\overline{K^-_k}$ 
and $\mathrm{rand}(\overline{K^-_k}, 1/\binom k2)\to 1/e$ as $k\to\infty$. 
Thus the asymptotic formula~\eqref{eq:8751} follows from the results in~\cites{FS, MMNT}. 
However, the exact 
values $\mathrm{ind}(K_5^-)=525/1024$, $\mathrm{ind}(K_6^-)=178200/13^5$, etc.\
implied by Theorem~\ref{thm:1643} are new. It seems likely that the recent work
by Liu, Pikhurko, Sharifzadeh, and Staden~\cite{LPSS} is relevant to the corresponding
stability problems.

%%%%%%%%%%%%%%%%%%%%%%%%%%%%%%%%%%
\subsection{Stars}\label{SUBSEC:star}
A second case of asymptotic equality in the edge-statistics conjecture occurs for stars. 
For every positive integer $t$ we denote the star with $t$ edges by $S_t$. As the case 
$S_1=K_2$ is trivial, we 
may assume~$t\ge 2$ in the sequel. A quick calculation shows that the induced $S_t$-density 
of a complete bipartite graph the sizes of whose vertex classes have roughly the ratio $1:t$
is $1/e+o_t(1)$, where again $o_t(1)$ tends to zero as $t$ tends to infinity.  

A precise formula for the inducibility of stars was discovered by 
Brown and Sidorenko~\cite{BS94} (see Theorem~\ref{THM:BS-inducibility-S_t} below).
Here we shall show that for small densities $x$ the values $I(S_t, x)$ of the upper 
bound function of the feasible region are realized by complete bipartite graphs. 
 
Toward this goal we consider for every real $x \in [0,1/2]$ 
a sequence $(B(n,x))_{n=1}^{\infty}$ 
of complete bipartite graphs with $v(B(n,x))=n$ for every $n\in \NN$ 
and $\lim_{n\to\infty}\rho(B(n,x)) = x$. The vertex classes of $B(n,x)$ have the sizes 
$\alpha n$ and $(1-\alpha)n$ for some $\alpha\in [0, 1/2]$ satisfying 
$\alpha(1-\alpha) = x/2 + o(1)$. 
Since $\rho(S_t, B(n, x))=(t+1)\bigl(\alpha(1-\alpha)^t+(1-\alpha)\alpha^t\bigr)+o_n(1)$,
we are led to the function $s_t\colon [0,1/2]\lra \RR$ defined by
\begin{align}\label{equ:St-left-side}
	s_{t}(x) 
	= 
	\lim_{n \rightarrow \infty} \rho(S_t, B(n,x)) 
	= 
	\frac{t+1}{2^{t}} x \left(\left(1-\sqrt{1-2x}\right)^{t-1}+\left(1+\sqrt{1-2x}	
		\right)^{t-1}\right)\,.
\end{align}

As we shall show in Section~\ref{SEC:proof-stars}, there is a unique point $x=x^\star(t)\in [0, 1/2]$ at which $s_t(x)$ attains its maximum. Moreover,
\[
	x^\star(2)
	= 
	x^\star(3) 
	= 
	\frac{1}{2}
	\quad \text{ and } \quad 
	\frac{2t}{(t+1)^2}
	<
	x^\star(t) 
	<
	\frac 2{t+1}
	\text { holds for } t\ge 4\,.
\]

Using Theorem~\ref{THM:inducibility-multi-partite-all} we determine $I(S_t,x)$ 
for $x\in [0,x^\star(t)]$.

\begin{theorem}\label{THM:inducibility-St-left-part}
If $t\ge 2$ is an integer and $x\in [0,x^\star(t)]$, then $I\left(S_{t},x\right) = s_{t}(x)$.
\end{theorem}

Notice that for $t=2$ this tells us $I(K_3^-, x)=3x/2$ for $x\in [0, 1/2]$, which 
follows from Theorem~\ref{THM:inducibility-S2} as well. 
It seems hard to determine $I(S_t,x)$ for $t\ge 3$ and $x\ge x^\star(t)$ (some remarks 
on this problem are given in Section~\ref{SEC:remarks}). 

For future reference it is convenient to extend the definitions of this subsection to 
the trivial case $t=1$ by setting $x^\star(1)=1/2$ and $s_1(x)=x$ for every $x\in [0, 1/2]$ 
(which is one half of the values one would obtain by plugging $t=1$ 
into~\eqref{equ:St-left-side}).
It is then still true that we have $I\left(S_{1},x\right) = s_{1}(x)$ 
for every $x\in [0,x^\star(1)]$ and that equality holds for the 
sequence~$(B(n,x))_{n=1}^{\infty}$ of bipartite graphs. 

\subsection{Complete bipartite graphs}\label{SUBSEC:complete-bipartite-graphs}
For positive integers $s$ and $t$ let $K_{s,t}$ denote the complete bipartite graph
whose vertex classes are of size $s$ and $t$. So $K_{1, t}=S_t$ is a star and 
it turns out that the calculation of $I(K_{s, t}, x)$ reduces to 
$I(S_{|s-t|+1}, x)$ for $x\in [0, x^\star(|s-t|+1)]$. 

\begin{theorem}\label{THM:feasible-reagion-Kst-upper-bound-x-small}
Let $t\ge s \ge 2$ be integers.
Then for every $x\in [0,1]$ we have
\begin{align}
I(K_{s,t}, x)
\le \frac{1}{2^{s-1}(t-s+2)}\binom{s+t}{s}x^{s-1}I(S_{t-s+1},x), \notag
\end{align}
and equality holds for $x\le x^\star(t-s+1)$.
In particular, for $x\in [0,x^\star(t-s+1)]$,
\begin{align}
I(K_{s,t}, x) =
\begin{cases}
 \frac{1}{2^{t}}\binom{2t}{t}x^{t} & {\rm if}\quad t=s,\\
 \frac{1}{2^{t}}\binom{s+t}{s}x^{s}\left(\left(1-\sqrt{1-2x}\right)^{t-s}+\left(1+\sqrt{1-2x}\right)^{t-s}\right) & {\rm if}\quad t > s.
\end{cases} \notag
\end{align}
\end{theorem}

The remainder of this subsection focuses on the case $s=t=2$. Observe that $K_{2, 2}=C_4$
is a four-cycle. Theorem~\ref{THM:feasible-reagion-Kst-upper-bound-x-small} 
yields $I(C_4, x)= 3x^2/2$ for every $x\in [0, 1/2]$, where equality is achieved by 
the sequence~$(B(n,x))_{n=1}^{\infty}$ of bipartite graphs. For $x\ge 1/2$ we believe 
that~$I(C_4, x)$ is related to the constructions for the clique density theorem (see 
Construction~\ref{con:cdt}).

\begin{conjecture}\label{CONJ:induceibity-K22}\label{conj:cdt}
For every real number $x\in [1/2,1]$ we have 
\begin{align}
I(C_{4},x) = \lim_{n\to \infty}\rho\left(C_{4},H^\star(n, x)\right)\,. \notag
\end{align}
\end{conjecture}

This conjecture predicts $I(C_4, 1-1/k)=3(k-1)/k^3$ for every integer $k\ge 2$ 
and our next result shows that this is indeed the case.

%\begin{observation}\label{OBS:construction-lower-bound-K2tb}
%We have $\lim_{n\to \infty}\rho(H_3(n,k,x)) = x$ for every integer $k \ge 3$ and real number $x\in ((k-2)/(k-1),(k-1)/k]$,
%and
%\begin{align}
%\lim_{n\to \infty}\rho\left(K_{2,t},H_3(n,k,(k-1)/k)\right) =
%\begin{cases}
%3(k-1)/k^{3} & \quad{\rm if}\quad t = 2 \\
%\binom{t+2}{2}(k-1)/k^{t+1} & \quad{\rm if}\quad t\ge 3
%\end{cases}\notag
%\end{align}
%for all integers $k\ge 3$.
%\end{observation}
%
%We have the following conjecture for $I(K_{2,2},x)$ and $x\in[1/2,1]$.
%
%
%A similar conjecture as above is not true for $K_{2,t}$ when $t\ge 3$ because of the following observation.
%
%\begin{observation}\label{OBS:construction-lower-bound-K2ta}
%For every integer $t\ge 3$ and every real number $x \in [1/2,1]$,
%\begin{align}
%\lim_{n\to \infty}\rho\left(K_{2,t},H_2(n,x)\right)
%= (t+2)(t+1)\left(\frac{1-x}{2}\right)^{t/2+1}. \notag
%\end{align}
%Moreover,
%\begin{align}
%(t+2)(t+1)\left(\frac{1-(k-1)/k}{2}\right)^{t/2+1}
%>
%\begin{cases}
%\binom{t+2}{2}(k-1)/k^{t+1} & {\rm if}\quad t = 3 \quad{\rm and}\quad k \ge 5, \\
%\binom{t+2}{2}(k-1)/k^{t+1} & {\rm if}\quad t \ge 4 \quad{\rm and}\quad k \ge 3.
%\end{cases}\notag
%\end{align}
%\end{observation}

\begin{figure}[htbp]
\centering
\begin{tikzpicture}[xscale=6,yscale=6]
\draw [->] (0,0)--(1.1,0);
\draw [->] (0,0)--(0,0.6);
\draw (0,0.5)--(1,0.5);
\draw (1,0)--(1,0.5);

\draw [line width=0.5pt,dash pattern=on 1pt off 1.2pt,domain=0:1] plot(\x,{3/8});
\draw [line width=0.5pt,dash pattern=on 1pt off 1.2pt] (1/2,0)--(1/2,3/8);
\draw [line width=0.5pt,dash pattern=on 1pt off 1.2pt] (2/3,0)--(2/3,2/9);
\draw [line width=0.5pt,dash pattern=on 1pt off 1.2pt] (0,2/9)--(2/3,2/9);
\draw [line width=0.5pt,dash pattern=on 1pt off 1.2pt] (0,9/64)--(3/4,9/64);
\draw [line width=0.5pt,dash pattern=on 1pt off 1.2pt] (3/4,0)--(3/4,9/64);

\draw[line width=0.7pt,color=sqsqsq,fill=sqsqsq,fill opacity=0.25]
(0.,0.)--(0.025,0.0009375)--(0.05,0.00375)--(0.075,0.0084375)--(0.1,0.015)--
(0.125,0.0234375)--(0.15,0.03375)--(0.175,0.0459375)--(0.2,0.06)--(0.225,0.0759375)--
(0.25,0.09375)--(0.275,0.113438)--(0.3,0.135)--(0.325,0.158438)--(0.35,0.18375)--
(0.375,0.210938)--(0.4,0.24)--(0.425,0.270938)--(0.45,0.30375)--(0.475,0.338438)--(0.5,0.375)--
(0.5,0.375)--(0.525,0.355359)--(0.55,0.334125)--(0.575,0.311578)--(0.6,0.288)--(0.625,0.263672)--
(0.65,0.238875)--(0.675,0.213891)--(0.7,0.189)--(0.725,0.164484)--(0.75,0.140625)--(0.775,0.117703)--
(0.8,0.096)--(0.825,0.0757969)--(0.85,0.057375)--(0.875,0.0410156)--(0.9,0.027)--(0.925,0.0156094)--
(0.95,0.007125)--(0.975,0.00182812)--(1.,0.);

\begin{scriptsize}
\draw [fill=uuuuuu] (1,0) circle (0.2pt);
\draw[color=uuuuuu] (1,0-0.07) node {$1$};
\draw [fill=uuuuuu] (0,0) circle (0.2pt);
\draw[color=uuuuuu] (0-0.05,0-0.05) node {$0$};
\draw [fill=uuuuuu] (0,3/8) circle (0.2pt);
\draw[color=uuuuuu] (0-0.05,3/8) node {$\frac{3}{8}$};
\draw [fill=uuuuuu] (0,2/9) circle (0.2pt);
\draw[color=uuuuuu] (0-0.05,2/9) node {$\frac{2}{9}$};
\draw [fill=uuuuuu] (3/4,9/64) circle (0.2pt);
\draw [fill=uuuuuu] (0,9/64) circle (0.2pt);
\draw[color=uuuuuu] (0-0.05,9/64) node {$\frac{9}{64}$};
\draw [fill=uuuuuu] (0,1/2) circle (0.2pt);
\draw[color=uuuuuu] (0-0.05,1/2) node {$\frac{1}{2}$};
\draw [fill=uuuuuu] (1/2,3/8) circle (0.2pt);
\draw[color=uuuuuu] (2/3,-0.07) node {$\frac{2}{3}$};
\draw [fill=uuuuuu] (2/3,2/9) circle (0.2pt);
\draw [fill=uuuuuu] (2/3,0) circle (0.2pt);
\draw[color=uuuuuu] (1/2,-0.07) node {$\frac{1}{2}$};
\draw [fill=uuuuuu] (1/2,0) circle (0.2pt);
\draw[color=uuuuuu] (3/4,-0.07) node {$\frac{3}{4}$};
\draw [fill=uuuuuu] (3/4,0) circle (0.2pt);
\end{scriptsize}
\end{tikzpicture}
\caption{$\Omega_{{\rm ind}}(C_{4})$ is contained in the shaded area above.}
\label{fig:feasible-region-C4}
\end{figure}
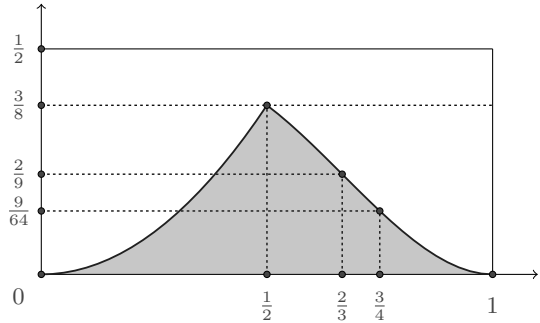

%The following result proves Conjecture~\ref{CONJ:induceibity-K22} for all integers $k \ge 2$ and $x=1-1/k$.

\begin{theorem}\label{THM:feasible-reagion-Kst-upper-bound-x-big}
If $x\in [1/2,1]$, then
\begin{align}
I(C_{4}, x)\le 3x(1-x)^2\,. \notag
\end{align}
Moreover, the bound is tight for all $x \in \left\{ (k-1)/k\colon k \in \NN \text{ and } k \ge 2 \right\}$.
\end{theorem}

%%%%%%%%%%%%%%%%%%%%%%%%%%%%%%%%%%%%%%%
%\subsection{The complete $3$-partite graph $K_{4}^{-}$}\label{SUBSEC:intro-K4-minus}
%We consider the graph $K_{4}^{-}$ which is obtained from $K_{4}$ by removing one edge.
%It is also the smallest complete $3$-partite graph whose feasible region is not known.
%
%Define $J\colon [0, 1]\lra \RR$ to be the piecewise linear function interpolating 
%between $J(0)=0$ and $J(1-1/r)=(r-1)(r-2)/r^3$ for $r\ge 4$. In other words, 
%$J(x)=x/8$ for $x\le 3/4$ and 
%\[
%	J(x)=\frac{(r-1)[r(r^2-3r-2)-(r^3-3r^2-5r-2)x]}{r^{2}(r+1)^{2}} 
%\]
%for $r\ge 4$ and $(r-1)/r\le x\le r/(r+1)$.
%
%\begin{theorem}\label{THM:feasible-region-K4-minus-all-x}
%We have
%\begin{align}
%I(K_{4}^{-},x) \le
%\begin{cases}
%3x^2/2 & {\rm if}\quad x\in [0,1/2]\\
%6J(x) & {\rm if}\quad x\in [1/2,1].
%\end{cases}\notag
%\end{align}
%Moreover, the bound is tight for 
%all $x\in \left\{(k-1)/k\colon k\in \NN \text{ and } k\ge 4\right\}$.
%\end{theorem}
%
%{\bf Remarks.}
%\begin{itemize}
%\item
%
%\item
%The $3$-partite Tur\'{a}n graphs show that $I(K_4^{-},2/3) \ge 4/9$ and we think this bound is tight.
%\item
%It seems much harder to determine $I(K_4^{-},x)$ for $x< 2/3$.  
%\end{itemize}
%

%%%%%%%%%%%%%%%%%%%%%%%%%%%%%%%%%%%%%%
\subsection*{Organization}
For every $x\in \{2, 3, 4, 5, 6\}$ the results stated in Subsection~1.$x$ 
are proved in Section~$x$. 
Section~\ref{SEC:remarks} contains further remarks and open problems.
%%%%%%%%%%%%%%%%%%%%%%%%%%%%%%%%%

%%%%%%%%%%%%%%%%%%%%%%%%%%%%%%%%%

\section{Proofs of general results}\label{SEC:proof-general-results}
We prove Theorem~\ref{thm:1347} and 
Proposition~\ref{PROP:lower-bound-is-zero}
%~\ref{PROP:general-upper-lower-bound-for-quantum-graphs}
in this section.
The following result is very similar to~\cite{LM21A}*{Proposition 1.3}.

\begin{proposition}\label{PROP:induced-feasible-region-is-closed}
For every quantum graph $Q$ the set $\Omega_{{\rm ind}}(Q)$ is closed. 
\end{proposition}

\begin{proof}
	Let $(x, y)\in [0, 1]\in\RR$ be a point belonging to the closure of $\Omega_{{\rm ind}}(Q)$.
	We are to exhibit a $Q$-good sequence $(G_n)_{n=1}^\infty$ of graphs realizing $(x, y)$.
	For every positive integer $n$ we first take a point $(x_n, y_n)\in \Omega_{{\rm ind}}(Q)$
	satisfying $|x_n-x|+|y_n-y|\le 1/n$ and then, using the definition 
	of $\Omega_{{\rm ind}}(Q)$, we take a graph $G_n$ such that 
	\[
		v(G_n)>n, \quad  |\rho(G_n)-x|<1/n, \quad \text{ and } \quad |\rho(Q, G_n)-y|<1/n. 
	\]
	The sequence $(G_n)_{n=1}^\infty$
	thus constructed satisfies $\lim_{n\to\infty} v(G_n)=\infty$, 
	$\lim_{n\to\infty} \rho(G_n)=x$, and $\lim_{n\to \infty} \rho(Q, G_n)=y$. 
	Since $Q$ has only finitely many constituents, some subsequence of $(G_n)_{n=1}^\infty$ 
	is $Q$-good. Every such subsequence realizes~$(x, y)$.
\end{proof}

Therefore the definitions of $i(Q,x)$ and $I(Q,x)$ rewrite as
\begin{align}
i(Q,x) = \min\left\{y\colon (x,y)\in \Omega_{{\rm ind}}(Q)\right\} 
\quad{\rm and}\quad
I(Q,x) = \max\left\{y\colon (x,y)\in \Omega_{{\rm ind}}(Q)\right\}. \notag
\end{align}

Next we show that $\Omega_{{\rm ind}}(Q)$ is determined by $i(Q,x)$ and $I(Q,x)$.

\begin{proposition}\label{PROP:feasible-region-determined-by-upper-lower-bounds}
Let $Q$ be a quantum graph, $x\in [0,1]$ and $y_1 < y_2$.
If $(x,y_1)\in \Omega_{{\rm ind}}(Q)$ and $(x,y_2)\in \Omega_{{\rm ind}}(Q)$,
then $(x,y)\in \Omega_{{\rm ind}}(Q)$ holds for all $y\in [y_1,y_2]$.
\end{proposition}

\begin{proof}[Proof of Proposition~\ref{PROP:feasible-region-determined-by-upper-lower-bounds}]
Fix $y\in [y_1, y_2]$. 
Let $\left(G'_{n}\right)_{n=1}^{\infty}$ be a $Q$-good sequence of graphs that 
realizes $(x,y_1)$,
and let $\left(G''_{n}\right)_{n=1}^{\infty}$ be a $Q$-good sequence of graphs that 
realizes $(x,y_2)$.
By a simple probabilistic argument we may assume $V(G'_{n}) = V(G''_{n}) = [n]$ 
for every $n\ge 1$.
We shall construct a sequence of graphs $\left(G_{n}\right)_{n=1}^{\infty}$ 
with $V(G_n)=[n]$ for every $n\ge 1$ that realizes~$(x, y)$.

For fixed $n\ge 1$ we consider a finite sequence of graphs $G^1_n, \dots, G^{m(n)}_n$
with common vertex set $[n]$ which interpolates between $G^1_n=G'_n$ and $G^{m(n)}=G''_n$ 
in the sense that 
\begin{enumerate}
	\item[$\bullet$] for $1\le m<m(n)$ the graph $G^{m+1}_n$ arises from $G^m_n$ by adding 
		or deleting a single edge, 
	\item[$\bullet$] and $\min\{\rho(G'_n), \rho(G''_n)\} \le \rho(G^m_n)\le 
		\max\{\rho(G'_n), \rho(G''_n)\}$ for every $m\in [m(n)]$.
\end{enumerate}	  

Due to the first bullet we have $\rho(Q, G^{m+1}_n)=\rho(Q, G^m_n)+o(1)$ for every
$m\in [m(n)-1]$. Combined with $\rho(Q, G^1_n)=y_1+o(1)$ and $\rho(Q, G^{m(n)}_n)=y_2+o(1)$
this proves that there exists some $k(n)\in [m(n)]$ such that the graph $G_n=G^{k(n)}_n$
satisfies $\rho(Q, G_n)=y+o(1)$. Owing to the second bullet we also have $\rho(G_n)=x+o(1)$. 
\end{proof}

Towards the continuity of $I(Q, x)$ we now establish the following lemma. 

\begin{lemma}\label{LEMMA:inequalities-I(H,x)}
For every quantum graph $Q$ there exist constants $\ell \ge 1$ and $C\ge 0$ such that
for all $x$, $x'$ with $0< x\le x' \le 1$ we have
\begin{align}\label{equ:increase-slow}
\frac{I(Q,x')}{(x')^{\ell}}
\le \frac{I(Q,x)}{x^{\ell}}
    + C \cdot \left(\left(\frac{1}{x}\right)^{\ell}-\left(\frac{1}{x'}\right)^{\ell}\right)\,.
\end{align}
\end{lemma}

\begin{proof}[Proof of Lemma~\ref{LEMMA:inequalities-I(H,x)}]
Fix $0< x \le x' \le 1$, set $\alpha = \left({x'}/{x}\right)^{1/2}-1$, 
and consider a $Q$-good sequence $\left(G'_n\right)_{n=1}^{\infty}$ that 
realizes $(x', I(Q, x'))$. Without loss of generality we may assume $v(G'_n) = n$ 
for every $n\ge 1$. Let $G_n$ be the graph which is the union of $G'_n$ and a set 
of~$\lfloor \alpha n \rfloor$ isolated vertices.
Since
\begin{align}
\rho\left(G_n\right)
= 
\frac{\rho\left(G'_n\right)\binom{n}{2}}{\binom{n+\lfloor \alpha n \rfloor}{2}}
\to 
\frac{x'}{(1+\alpha)^2} = x 
\quad{\rm as}\quad n\to \infty\,, \notag
\end{align}
we have 
\begin{equation}\label{eq:2039}
	I(Q, x)
	\ge
	\limsup_{n\to\infty} \rho(Q, G_n)\,. 
\end{equation}

To estimate the right side we write $Q = \sum_{i\in P}\lambda_iF_i + \sum_{j\in N}\lambda_jF_j$
with $\lambda_i > 0$ for $i\in P$ and $\lambda_j < 0$ for $j\in N$.
Set $\ell_i = v(F_i)$ for every $i\in P\cup N$ and $\ell = \max\{\ell_i/2\colon i\in P\cup N\}$.
For every $i\in P$ the fact that $G'_n$ is a subgraph of $G_n$ yields
\begin{align}\label{equ:lower-bound-fx}
\rho\left(F_i, G_n\right)
\ge 
\frac{\rho\left(F_i,G'_n\right)\binom{n}{\ell_i}}{\binom{n+\lfloor \alpha n \rfloor}{\ell_i}}
\ge 
\frac{\rho\left(F_i,G'_n\right)}{(1+\alpha)^{\ell_i}}
= 
\frac{\rho\left(F_i,G'_n\right)}{(x'/x)^{\ell_i/2}}
\ge 
\frac{\rho\left(F_i,G'_n\right)}{(x'/x)^{\ell}}.
\end{align}
For $j\in N$ we use that every induced copy of $F_j$ in $G_n$ is either already 
contained in $G'_n$ or involves one of the new isolated vertices, which implies 
\[    
	\rho\left(F_j, G_n\right)
	\le 
	\frac{\rho\left(F_j,G'_n\right)\binom{n}{\ell_j}+\alpha n\cdot \binom{n+\lfloor \alpha n \rfloor}{\ell_j-1}}{\binom{n+\lfloor \alpha n \rfloor}{\ell_j}} 
	\le
	\rho(F_j,G'_n)+ \frac{\ell_j \cdot \alpha}{1+\alpha} + o_n(1) \,.
\]
Taking into account that 
\[
	\rho(F_j,G'_n)
	\le
	\frac{\rho\left(F_j,G'_n\right)}{(x'/x)^{\ell}}
	+
	\left(1-\left(\frac{x}{x'}\right)^{\ell}\right)
\]
and 
\[
	\frac{\alpha}{1+\alpha} 
	=
	1-\left(\frac{x}{x'}\right)^{1/2}
	\le
	1-\left(\frac{x}{x'}\right)^{\ell}
\]
we obtain
\[
	\rho\left(F_j, G_n\right)
	\le
	\frac{\rho(F_j, G'_n)}{(x'/x)^{\ell}} 
		+ (\ell_j+1)\left(1-\left(\frac{x}{x'}\right)^{\ell}\right)+ o_n(1)\,.
\]

Combined with~\eqref{equ:lower-bound-fx} this entails
\begin{align*}
	\rho\left(Q, G_n\right)
	& =  
	\sum_{i\in P}\lambda_i\rho\left(F_i,G_n\right) 
		+ \sum_{j\in N}\lambda_j\rho\left(F_j,G_n\right) \\
	& \ge 
	\sum_{i\in P}\lambda_i\frac{\rho\left(F_i,G'_n\right)}{(x'/x)^{\ell}}
      + \sum_{j\in N}\lambda_j\left(\frac{\rho\left(F_j,G'_n\right)}{(x'/x)^{\ell}}
      + (\ell_j+1)\left(1-\left(\frac{x}{x'}\right)^{\ell}\right)\right) - o_n(1)\\
	& \ge 
	\frac{\rho(Q, G'_n)}{(x'/x)^{\ell/2}} 
		 - C \cdot \left(1-\left(\frac{x}{x'}\right)^{\ell}\right)- o_n(1)\,, 
\end{align*}
where $C = \sum_{j\in N}(-\lambda_{j})(\ell_j+1) \ge 0$.
Now~\eqref{eq:2039} reveals
\begin{align*}
I(Q, x)
\ge 
\frac{I(Q, x')}{(x'/x)^{\ell}} - C \cdot \left(1-\left(\frac{x}{x'}\right)^{\ell}\right)
\end{align*}
and upon multiplying both sides by $x^{-\ell}$ the claim follows. 
\end{proof}%lemma

For later use we record the following consequence.

\begin{corollary}\label{CORO:I(Q,x)-continuous}
Given a quantum graph $Q$ and $x'\in [0, 1]$, $\eps>0$, there exists some $\delta>0$
such that $I(Q, x)>I(Q, x')-\eps$ holds for all $x\in [0, x')$ with $|x-x'|<\delta$. \qed
\end{corollary}

Now we are ready to prove the main result of Subsection~\ref{SUBSEC:general-results-intro}.

\begin{proof}[Proof of Theorem~\ref{thm:1347}]
Given a quantum graph $Q$ the formula 
\[
	\Omega_{{\rm ind}}(Q)
	=
	\bigl\{(x, y)\in [0, 1]\times \RR\colon i(Q, x)\le y\le I(Q, x)\bigr\}
\]
follows immediately from Proposition~\ref{PROP:feasible-region-determined-by-upper-lower-bounds}.
Now, due to Fact~\ref{FACT:complement-minus-inducibility}\ref{it:13a} it suffices to show
that~$I(Q, x)$ is continuous and almost everywhere differentiable.

Let $\ell\ge 1$, $C\ge 0$ be the constants provided by Lemma~\ref{LEMMA:inequalities-I(H,x)}.
Owing to~\eqref{equ:increase-slow} the function $F\colon (0, 1]\longrightarrow\RR$ defined 
by $F(x) = \left(I(Q,x)+C\right)/x^{\ell}$ is decreasing. It follows that $F$ is almost 
everywhere differentiable and that for every $x\in (0, 1]$ the left-sided limit 
$\lim_{x\to x_0^{-}}F(x)$ exists. Consequently, the function $I(Q, x)$ has the same properties. 

Let us show next that $I(Q, x)$ is left-continuous. Given an arbitrary $x_0\in (0, 1]$
we already know that the limit $y_0 = \lim_{x\to x_0^{-}}I(Q,x)$ exists.  
Proposition~\ref{PROP:induced-feasible-region-is-closed} 
yields $(x_0, y_0)\in \Omega_{{\rm ind}}(Q)$, whence $I(Q, x_0)\ge y_0$.
But $I(Q, x_0) > y_0$ would contradict Corollary~\ref{CORO:I(Q,x)-continuous}
and thus we have indeed $I(Q, x_0) = y_0$.
By Fact~\ref{FACT:complement-minus-inducibility}\ref{it:13b} the function 
$I(Q, x)=I(\overline{Q}, 1-x)$ is right-continuous as well. This concludes the proof. 
\end{proof}%theorem

\begin{proof}[Proof of Proposition~\ref{PROP:lower-bound-is-zero}]
For every $n\in\NN$ and $x\in [0, 1]$ we let $H'(n, x)$ denote the $n$-vertex graph  
consisting of a clique of order $\lfloor x^{1/2}n\rfloor$ and $n-\lfloor x^{1/2}n\rfloor$
isolated vertices. Moreover, we set $H''(n, x)=\overline{H'(n, 1-x)}$. Notice that 
$\lim_{n\to\infty} \rho(H'(n, x))=\lim_{n\to\infty} \rho(H''(n, x))=x$ holds for 
every $x\in [0, 1]$. 

Now suppose that $F$ is a graph which is neither complete nor empty. 
If $F$ has no isolated vertex, then $\rho(F, H'(n, x))=0$ holds for all $n\in \NN$ 
and $x\in [0, 1]$, which leads to $i(F, x)=0$. If $F$ has an isolated vertex 
we get the same conclusion from $\rho(F, H''(n, x))=0$.
\end{proof}

%%%%%%%%%%%%%%%%%%%%%%%%%%%%%%%%%
\section{Proof for complete multipartite graphs}\label{SEC:proof-complte-multipartite}
We prove Theorem~\ref{THM:inducibility-multi-partite-all} in this section.
The following result of Schelp and Thomason~\cite{ST98} will be useful in our argument.

\begin{theorem}[Schelp-Thomason \cite{ST98}]\label{THM:ST98}
Let $Q = \sum_{i\in [m]}\lambda_i F_i$ be a quantum graph whose constituents are 
complete multipartite graphs and let $n\in\NN$. If every $F_i$ with $\lambda_i<0$ is 
complete, then 
among all $n$-vertex graphs $G$ maximizing $\rho(Q,G)$ there is a complete multipartite one. 
\end{theorem}

\begin{definition}
Suppose that $H\colon [0,1] \to \RR$ is a concave function
and $L\colon [0,1] \to \RR$ is a linear function.
We say $L$ is a {\it tangent line} of $H$ at $x_0 \in [0,1]$ if $L(x) \ge H(x)$  
holds for $x\in [0,1]$ with equality for $x = x_0$.
\end{definition}

It is easy to see that for every concave function $F\colon [0,1] \to \RR$
and every $x_0\in (0,1)$ there always exists a (not necessarily unique) tangent line of $F$ at $x_0$.

\begin{proof}[Proof of Theorem~\ref{THM:inducibility-multi-partite-all}]
By Fact~\ref{FACT:complement-minus-inducibility}\ref{it:13a} it suffices to 
show part~\ref{it:110b}.
Let $Q = \sum_{i\in [m]}\lambda_i F_i$ be a quantum graph whose constituents are 
complete multipartite graphs such that every~$F_i$ with $\lambda_i < 0$ is complete.
For brevity we set $H(x) = \capp\left(M(Q,\cdot)\right)(x)$  
for every $x\in[0,1]$. Clearly 
\[
	H(0)=M(Q, 0)=\lim_{n\to\infty}\rho(Q, \overline{K_n})=I(Q, 0)
\]
and a similar argument shows $H(1)=I(Q, 1)$. So it remains to prove $H(x_0)\ge I(Q, x_0)$
for every $x_0\in (0,1)$. To this end we choose a tangent line $L(x)=kx+p$ of $H$ at $x_0$,
so that  
\begin{align}\label{equ:F(x)-less-than-L(x)}
	H(x)\le kx+p 
	\quad \text{ for all } x\in [0, 1] 
	\quad \text{ and } \quad 
	H(x_0)=kx_0+p\,.
\end{align}

Now let $\left(G_n\right)_{n=1}^{\infty}$ be a sequence of graphs that 
realizes $(x_0, I(Q, x_0))$. By Theorem~\ref{THM:ST98} applied to the quantum 
graph $Q^\star=Q -k K_2$ there exists for every $n\ge 1$ a multipartite $n$-vertex 
graph $G'_n$ such that $v(G'_n)=v(G_n)$ and 
\begin{equation}\label{eq:0203}
	\rho(Q, G_n)-k\rho(G_n)
	=
	\rho(Q^\star, G_n)
	\le
	\rho(Q^\star, G'_n)
	=
	\rho(Q, G'_n)-k\rho(G'_n)\,.
\end{equation}
By passing to a subsequence of $\left(G'_n\right)_{n=1}^{\infty}$ we may assume that 
the limits $x_1=\lim_{n\to \infty}\rho(G'_n)$ and $y_1=\lim_{n\to \infty}\rho(Q, G'_n)$
exist. Due to the definition of $M(Q, x_1)$ and~\eqref{equ:F(x)-less-than-L(x)} 
we have 
\[
	y_1
	\le 
	M(Q, x_1)
	\le 
	H(x_1)
	\le 
	kx_1+p
\]
and taking the limit $n\to\infty$ in~\eqref{eq:0203} it follows that 
\[
	I(Q, x_0)-kx_0
	\le 
	y_1-kx_1
	\le 
	p\,.
\]
Together with~\eqref{equ:F(x)-less-than-L(x)} this leads to the desired estimate 
$I(Q, x_0)\le kx_0+p=H(x_0)$.
\end{proof}%THM
%%%%%%%%%%%%%%%%%%%%%%%%%%%%%%%%%

\section{Proofs for almost complete graphs}\label{SEC:proof-Kt-}

In this section we prove Theorems~\ref{THM:inducibility-S2} and~\ref{thm:1643}
as well as Proposition~\ref{prop:1648}.

\subsection{Cherries}\label{SUBSEC:S2-full}
We begin with the proof of Theorem~\ref{THM:inducibility-S2}.
Consider a graph $G=(V, E)$ with $|V|=n$ vertices. Counting the number of pairs
$(\{x, y\}, z)\in E\times V$ with $z\ne x, y$ in two different ways, we obtain
\[
	(n-2)|E|=N(\overline{K_3^-}, G)+2N(K_3^-, G)+3N(K_3, G)\,.
\]
Dividing by $2\binom n3$ and rearranging we deduce
\[
	\rho(K_3^-, G)
	=
	\frac 32\bigl(\rho(K_2, G)-\rho(K_3, G)\bigr)-\frac12 \rho(\overline{K_3^-}, G)\,.
\]

Therefore the clique density theorem yields for every $x\in [0, 1]$ 
the upper bound $I(K_3^-, x)\le \frac32(x-g_3(x))$. 
Moreover, for every $x\in [0, 1]$
the sequence of multipartite graphs $(H^\star(n,x))_{n=1}^\infty$ 
is $\overline{K_3^-}$-free and establishes the lower bound $I(K_3^-, x)\ge \frac32(x-g_3(x))$. 

\subsection{Piecewise linear upper bounds}
Roughly speaking we show in this subsection that a concave piecewise linear function 
is an upper bound on $I(K_t^-, x)$ if it respects the constraints coming from Tur\'an graphs.

\begin{lemma}\label{lem:7240}
	Suppose that an integer $s\ge 1$ and real numbers $\lambda$, $\mu$ have the 
	property that 
	\begin{equation}\label{eq:7147}
		\frac{1}{r^{s+1}}\binom{r-1}{s}\le \lambda\frac{r-1}{2r}+\mu
	\end{equation}
	holds for every positive integer $r$. 
	If $m\ge 1$ and $(\alpha_1, \dots, \alpha_m)\in \Delta_{m-1}$, then 
	\[
		\sum_{i=1}^m\sum_{W\in\binom{[m]\setminus\{i\}}{s}} 
			\alpha_i^2\prod_{j\in W}\alpha_j
		\le
		\lambda \sum_{\{i, j\}\in\binom{[m]}{2}} \alpha_i\alpha_j +\mu\,.
	\]
\end{lemma}

\begin{proof}
	Assume for the sake of contradiction that this fails and let $m$ denote the least positive
	integer for which there exists a counterexample. 
	%Next let $I\subseteq [m]$ be a maximal 
	%set with the property that there is a counterexample 
	%$(\alpha_1, \dots, \alpha_m)\in \Delta_{m-1}$ satisfying $\alpha_i=1/m$ for every $i\in I$.
	Appealing to a theorem of Weierstra\ss, 
	we pick a point $(\alpha^\star_1, \dots, \alpha^\star_m)\in \Delta_{m-1}$ such that 
	%$\alpha^\star_i=1/m$ holds for every $i\in I$ and subject to this 
	the difference 
	\[
		\Phi= \sum_{i=1}^m\sum_{W\in\binom{[m]\setminus\{i\}}{s}} 
			(\alpha^\star_i)^2\prod_{j\in W}\alpha^\star_j
			-
			\lambda \sum_{\{i, j\}\in\binom{[m]}{2}} \alpha^\star_i\alpha^\star_j
	\]
	is maximal. Due to our indirect assumption we know $\Phi>\mu$. The case $r=m$ 
	of~\eqref{eq:7147} reveals that $\alpha^\star_1=\dots=\alpha^\star_m=1/m$ is false. 
	Therefore, we have $m\ge 2$ and and for reasons of symmetry 
	we may assume that $\alpha^\star_1<\alpha^\star_2$.  
	
	Given two real numbers $\alpha_1, \alpha_2\ge 0$ satisfying 
	\[
		\alpha_1+\alpha_2=\alpha^\star_1+\alpha^\star_2
	\]
	we write $\Phi(\alpha_1, \alpha_2)$ for the result of 
	replacing $\alpha^\star_1$, $\alpha^\star_2$
	in the above formula for $\Phi$ by $\alpha_1$, $\alpha_2$. 
	So~$\Phi(\alpha^\star_1, \alpha^\star_2)=\Phi$ and there are constants $c_1, \ldots, c_5$
	depending only on $\alpha^\star_3, \dots, \alpha^\star_m$, and $\lambda$ such that 
	\[
		 \Phi(\alpha_1, \alpha_2)=c_1+c_2(\alpha_1+\alpha_2)+c_3(\alpha^2_1+\alpha^2_2)
		 	+c_4\alpha_1\alpha_2+c_5(\alpha_1+\alpha_2)\alpha_1\alpha_2\,.
	\]
	Since $\alpha_1+\alpha_2$ is constant and $\alpha^2_1+\alpha^2_2$, $2\alpha_1\alpha_2$
	add up to the constant $(\alpha^\star_1+\alpha^\star_2)^2$, it follows that there are 
	constants~$c_6$, $c_7$ such that 
	\[
		 \Phi(\alpha_1, \alpha_2)=c_6\alpha_1\alpha_2+c_7\,.
	\]
	If $c_6 \ne 0$ we can find a real number $\xi\ne 0$ such that $|\xi|$ is very small 
	and $\Phi(\alpha^\star_1+\xi, \alpha^\star_2-\xi)>\Phi$ contradicts the maximality 
	of $\Phi$. So $c_6=0$ and $\Phi(\alpha_1, \alpha_2)=c_7=\Phi$ is constant. 
	But now $\Phi(\alpha^\star_1+\alpha^\star_2, 0)=\Phi$ contradicts 
	the minimality of~$m$. This completes the proof. 
\end{proof}

\begin{lemma}\label{lem:4100}
	Suppose that $t\ge 3$ and that $f\colon [0, 1]\lra\RR$ is a piecewise linear concave function. 
	If for every positive integer $r$ we have 
	\begin{equation}\label{eq:7254}
		f(1-1/r)
		\ge 
		\binom{t}{2}\frac{(r-1) \cdots (r-(t-2))}{r^{t-1}}\,,
	\end{equation}
	then $I(K_t^-, x)\le f(x)$
	holds for every $x\in [0, 1]$. 
\end{lemma}

\begin{proof}
	Since $f$ is the pointwise minimum of a family of linear functions, it suffices to 
	deal with the case that $f(x)=\lambda x+\mu$ is itself linear. 
	By Theorem~\ref{THM:inducibility-multi-partite-all}\ref{it:110b} it is enough to show 
	$M(K_t^-, x)\le \lambda x+\mu$ for every $x\in [0, 1]$. 
	We shall establish the more precise estimate that every complete multipartite graph $G$ 
	on $n$ vertices satisfies 
	\begin{equation}\label{eq:1217}
		N(K_t^-, G)\le (2\lambda |E(G)|+\mu n^2)n^{t-2}/t!\,.
	\end{equation}

	Let $a_1, \dots, a_m$ be the sizes of the vertex classes of $G$ and set $\alpha_i=a_i/n$
	for every $i\in [m]$. Now $\sum_{i=1}^m \alpha_i=1$ and 
	\[
		N(K^-_t, G)
		=
		\sum_{i=1}^m\binom{a_i}2\sum_{W\in\binom{[m]\setminus\{i\}}{t-2}} \prod_{j\in W}a_j
		\le 
		\frac{n^t}2
		\sum_{i=1}^m\alpha^2_i\sum_{W\in\binom{[m]\setminus\{i\}}{t-2}} \prod_{j\in W}\alpha_j
	\]
	and, therefore, instead of~\eqref{eq:1217} it suffices to show 
	\[
		\sum_{i=1}^m\alpha^2_i\sum_{W\in\binom{[m]\setminus\{i\}}{t-2}} \prod_{j\in W}\alpha_j
		\le
		\frac{4\lambda}{t!}\sum_{\{i, j\}\in \binom{[m]}2} \alpha_i\alpha_j+\frac{2\mu}{t!}\,.
	\]
	By Lemma~\ref{lem:7240} applied to $t-2$, $4\lambda/t!$, $2\mu/t!$ here in place of 
	$s$, $\lambda$, $\mu$ there this inequality follows from the fact that 
	\[
		\frac{1}{r^{t-1}}\binom{r-1}{t-2}
		\le 
		\frac{4\lambda}{t!}\cdot \frac{r-1}{2r} + \frac{2\mu}{t!}
		=
		\frac{2f(1-1/r)}{t!}
	\]
	holds for every $r\ge 1$, which is in turn equivalent to the hypothesis~\eqref{eq:7254}. 
\end{proof}

\subsection{Precise calculations} 
Fix an integer $t\ge 4$. Our next goal is to show that the function $h_t$ introduced 
in Theorem~\ref{thm:1643} satisfies the assumptions of Lemma~\ref{lem:4100}. To this end 
we set $A_r=\binom t2\frac{(r-2)_{t-3}}{r^{t-2}}$ for every integer $r\ge 2$. 

\begin{lemma}
	Let $t\ge 4$ and $r\ge t-1$ be integers. 
	\begin{enumerate}[label=\alabel]
		\item\label{it:4a} If $r\le (3t^2-5t-4)/6$, then $A_{r-1}<A_r$. 
		\item\label{it:4b} If $r=(3t^2-5t-2)/6$, then $A_{r-1}<A_r$ or $A_{r-1}>A_r$ 
			holds depending on whether $t\le 20$ or $t>20$.
		\item\label{it:4c} If $r\ge (3t^2-5t)/6$, then $A_{r-1}>A_r$ .
	\end{enumerate}
	In particular, there exists a unique integer $k\ge t-2$ satisfying 
	$A_k=\max\{A_r\colon r\ge t-2\}$, namely $k=k(t)$. 
\end{lemma}

\begin{proof}
	One confirms easily that 
	\begin{equation}\label{eq:431}	
		A_{r-1}<A_r
		\,\,\, \Longleftrightarrow \,\,\,
		1-\frac{t-1}r < \Bigl(1-\frac 1r\Bigr)^{t-2}\Bigl(1-\frac 2r\Bigr)\,.
	\end{equation}
	To estimate the first factor on the left side we shall use the approximations 
	\begin{equation}\label{eq:432}				
		\sum_{i=0}^3 \frac{(-1)^i}{r^i}\binom{t-2}{i}		
		\le
		\left(1-\frac 1r\right)^{t-2}  
		\le
		\sum_{i=0}^4 \frac{(-1)^i}{r^i}\binom{t-2}{i}\,.	
	\end{equation}

	Let us now prove part~\ref{it:4a}. So we have $r\le (3t^2-5t-4)/6$ and 
	by~\eqref{eq:431},~\eqref{eq:432} it suffices to show 
	\[
		1-\frac{t-1}r 
		< 
		\sum_{i=0}^3 \frac{(-1)^i}{r^i}\binom{t-2}{i}\Bigl(1-\frac 2r\Bigr)\,,
	\]
	which rewrites as 
	\[
		r\left(\frac{(t+1)(t-2)}2-r\right)
		>
		\frac{(t+2)(t-2)(t-3)}{6}\,.
	\]
	Due to $(t-1)/3\le r\le (3t^2-5t-4)/6$ we have indeed 	
	\[
		r\left(\frac{(t+1)(t-2)}2-r\right) 
		\ge
		\frac{3t^2-5t-4}6 \cdot \frac{t-1}3
		\ge
		\frac{t(t-1)(t-2)}{6}
		>
		\frac{(t+2)(t-2)(t-3)}{6}\,,
	\]
	which concludes the proof of~\ref{it:4a}.
	
	Next we observe that~\eqref{eq:431} and~\eqref{eq:432} also yield a sufficient 
	condition for $A_{r-1}>A_r$, namely the estimate 
	\[
		1-\frac{t-1}r 
		> 
		\sum_{i=0}^4 \frac{(-1)^i}{r^i}\binom{t-2}{i}\Bigl(1-\frac 2r\Bigr)\,.
	\]
	Simplifying this inequality we arrive at the implication  
	\begin{align}\label{eq:433}
		\frac{(t+2)(t-2)(t-3)}6 > r\left(\frac{(t+1)(t-2)}2-r\right)+\frac{t+3}{4r}\binom{t-2}3
		   \,\,\, \Longrightarrow A_{r-1}>A_r\,.
	\end{align}

	Thus part~\ref{it:4c} follows from the fact that $r\ge (3t^2-5t)/6>(t-2)(t+3)/6$ 
	implies 
	\begin{align*}
		r\left(\frac{(t+1)(t-2)}2-r\right) + \frac{t+3}{4r}\binom{t-2}3
		&<
		\frac{3t^2-5t}6\cdot \frac{t-3}3 + \frac{(t-3)(t-4)}{4} \\
		&<
		\frac{(t+2)(t-2)(t-3)}6\,.
	\end{align*}

	We proceed with the proof of part~\ref{it:4b}. Now $r=(3t^2-5t-2)/6$ is an integer, 
	which requires
	$t\equiv 2\pmod{3}$. Direct calculations based on~\eqref{eq:431} 
	show $A_{r-1}<A_r$ for $t\in\{5, 8, 11, 14, 17, 20\}$
	and $A_{r-1}>A_r$ for $t=23$. 
	As soon as $t\ge 26$ we have $8(t-8)r>3(t+3)(t-3)(t-4)$ and hence  
	\begin{align*}
		r\left(\frac{(t+1)(t-2)}2-r\right) + \frac{t+3}{4r}\binom{t-2}3
		&<
		\frac{3t^2-5t-2}6\cdot \frac{t-2}3 + \frac{(t-2)(t-8)}{9} \\
		&=
		\frac{(t+2)(t-2)(t-3)}6\,,
	\end{align*}
	which in view of~\eqref{eq:433} concludes the discussion of~\ref{it:4b}. 
	
	Finally,~\ref{it:4a}\,--\,\ref{it:4c}
	together imply 
	\[
		A_{t-2}<A_{t-1}< \dots < A_{k(t)}
		\quad \text{ and } \quad 
		A_{k(t)} > A_{k(t)+1} > \dots\,,
	\]
	whence $A_{k(t)} = \max\{A_r\colon r\ge t-2\}$. 
\end{proof}

\begin{lemma}
	We have $I(K_t^-, x)\le h_t(x)$ for every $x\in [0, 1]$.
\end{lemma}

\begin{proof}
   For later use we observe that the number $k=k(t)$ satisfies 
	\begin{equation}\label{eq:4754}
		k\ge \frac{t(t-2)}2\,.
	\end{equation}
	Indeed, if $t\ne 5, 8, 11, 14, 17, 20$, then 
	$k-t(t-2)/2=\lceil (t-8)/6\rceil \ge \lceil -2/3\rceil=0$ and in the remaining cases 
	we have $k-t(t-2)/2=(t-2)/6\ge 0$. 
	
	Next we show 
	\[
		h_t(1-1/r)
		\ge 
		\binom{t}{2}\frac{(r-1) \cdots (r-(t-2))}{r^{t-1}}
	\]
	for every positive integer $r$. The cases $r\le t-2$ and $r\ge k$ are clear. 
	Now suppose that $t-1\le r < k$. Since $1-1/r\le 1-1/k$ and $h_t(x)=A_{k} \cdot x$
	for all $x\in [0, 1-1/k]$ we have 
	\[
		h_t(1-1/r)
		=
		A_{k}\cdot (1-1/r)
		\ge 
		A_r\cdot (1-1/r)=\binom{t}{2}\frac{(r-1) \cdots (r-(t-2))}{r^{t-1}}\,,
	\]
	as desired. 
	
	According to Lemma~\ref{lem:4100} it only remains to show that $h_t$ is concave.
	Now $A_{k}>A_{k+1}$ rewrites as 
	\[
		\frac{h_t(1-1/k)}{1-1/k}
		>
		\frac{h_t(1-1/(k+1))}{1-1/(k+1)}
	\]
	and, therefore, $h_t$ is concave in some sufficiently small neighbourhood around $x=1-1/k$.
	Define $F\colon \bigl[0, \frac 1{k}\bigr]\lra \RR$ by $F(x)=x\prod_{i=1}^{t-2}(1-ix)$.
	Since $h_t(1-1/r)=\binom t2 F(1/r)$ holds for every $r\ge k$, it suffices to show that 
	$F$ is concave. If $x\in [0, 1/k]$, then
	\[
		\sum_{i=1}^{t-2}\frac i{1-ix}
		\le 
		\frac{1+\dots+(t-2)}{1-(t-2)x}
		\overset{\eqref{eq:4754}}{\le}
		\frac{(t-2)(t-1)}{2(1-2/t)}
		=
		\frac{(t-1)t}{2}
		<
		\frac 2x
	\]      
	and thus
	\[
		\frac{F''(x)}{F(x)}
		=
		\sum_{1\le i<j\le t-2}\frac{ij}{(1-ix)(1-jx)}	-\frac 1x \sum_{i=1}^{t-2}\frac i{1-ix}
		<
		\frac12\Bigl(\sum_{i=1}^{t-2}\frac i{1-ix}\Bigr)^2 -\frac 1x \sum_{i=1}^{t-2}\frac i{1-ix}
		\le 
		0\,,
	\]
	which proves that $F$ is indeed concave. 
\end{proof}

The only part of Theorem~\ref{thm:1643} still lacking verification is~\eqref{eq:8648}.
Setting $B_r=\binom{t}{2}\frac{(r-1)_{t-2}}{r^{t-1}}$
for every $r\ge t-2$ and $f=\lceil (t-2)(3t+1)/6\rceil$ we are to 
show $B_f=\max\{B_r\colon r\ge t-2\}$. It turns out that this holds in the following 
slightly stronger form. 

\begin{lemma}\label{lem:8351}
	We have $0=B_{t-2}<B_{t-1}<\dots<B_f$ and $B_f>B_{f+1}>\dots$.
\end{lemma}

\begin{proof}
	First we show $B_{r-1}<B_r$ for every integer $r\in [t-1, f]$. 
	The fact that $(t-2)(3t+1)$ is even yields $f\le (t-2)(3t+1)/6+2/3=(t-1)(3t-2)/6$,
	whence	
	\[
		\frac{t-1}3\le r\le \binom t2-\frac{t-1}3\,.
	\]
	For this reason we have
	\[
		r\left(\binom t2 - r\right)
		\ge
		\frac{t-1}3 \left(\binom t2 - \frac{t-1}3\right)
		>
		\binom t3\,,
	\]
	which rewrites as 
	\[
		1-\frac{t-1}r
		<
		1-\frac{t}r+\binom t2\frac 1{r^2}-\binom t3\frac 1{r^3}\,.
	\]
	As the right side is at most $(1-1/r)^t$, this proves 
	\[
		1
		< 
		\frac{(r-1)^t}{r^{t-1}(r-(t-1))}
		=
		\frac{B_r}{B_{r-1}}\,,
	\]
	as desired. 
	
	Next we show $B_{r-1}>B_r$ for every $r\ge f+1$. 
	Due to $r\ge (3t^2-5t+4)/6>\frac 12\binom t2$ we have 
	\[%begin{equation}\label{eq:4229}
		r\left(\binom t2 - r\right)
		<
		\frac{3t^2-5t+4}6 \cdot \frac{t-2}3
		=
		\binom t3 - \frac{(t-2)^2}9\,.
	\]%end{equation}
	Moreover, $r\ge t(t-2)/2$ implies 
	\[
		\binom t4 \cdot \frac 1{r}
		<
		\frac{(t-1)(t-3)}{12}
		<
	   \frac{(t-2)^2}{9}\,.
	 \]
	 Adding the previous two estimates we obtain 
	 \[
	 		r\left(\binom t2 - r\right)+\binom t4 \cdot \frac 1{r}
			<
			\binom t3\,,
	\]
	which rewrites as 
	\[
		1-\frac{t-1}r
		>
		1-\frac{t}r+\binom t2\frac 1{r^2}-\binom t3\frac 1{r^3}+\binom t4\frac 1{r^4}\,.
	\]
	As the right side is an upper bound on $(1-1/r)^t$ we can conclude 
	\[
		1
		>
		\frac{(r-1)^t}{r^{t-1}(r-(t-1))}
		=
		\frac{B_r}{B_{r-1}}\,. \qedhere
	\]
\end{proof}

\subsection{More on  \texorpdfstring{$K_4^-$}{K4-}} 
Our last result on $\Omega_{\mathrm{ind}}(K_4^-)$, Proposition~\ref{prop:1648}, 
is an immediate consequence of the following result. 

\begin{lemma}
	Every graph $G$ 
	satisfies $N(K_4^-, G)\le \frac12 \binom{|E(G)|}2$.
\end{lemma}

\begin{proof}
	Notice that an abstract $K_4^-$ has two perfect matchings. Now  
	with every induced copy of $K_4^-$ in~$G$ we associate its two
	perfect matchings, viewed as members of $\binom{E(G)}2$. We are 
	thereby considering $2N(K_4^-, G)$ pairs of edges of $G$. Since 
	every pair $\{e, f\}\in \binom{E(G)}2$ can be associated to at most 
	one copy of $K_4^-$ in $G$ (namely the copy induced by $e\cup f$, 
	if it exists), this proves the claim. 
\end{proof}

%\begin{cor}
%	We have $I(K_4^-, x)\le 3x^2/2$ for every $x\in [0, 1]$. \qed
%\end{cor}

\section{Proofs for stars}\label{SEC:proof-stars}

In this section we prove Theorem~\ref{THM:inducibility-St-left-part}.
Recall from Section~\ref{SUBSEC:star} that for every integer $t\ge 3$ and every real 
$x\in[0,1/2]$ we defined 
\begin{align}
s_{t}(x) = \frac{t+1}{2^{t}} x \left(\left(1-\sqrt{1-2x}\right)^{t-1}+\left(1+\sqrt{1-2x}\right)^{t-1}\right). \notag
\end{align}

We commence by showing that there is a unique $x^\star(t)\in [0,1/2]$, where the 
function $s_t$ attains its maximum. For $t=3$ we have $s_3(x)=2x(1-x)$ and, hence,  $x^\star(3)=1/2$ is as desired. The case $t\ge 4$ is addressed by the next lemma. 

\begin{lemma}
	For $t\ge 4$ there exists a unique real 
	$x^\star(t)\in \bigl(\frac{2t}{(t+1)^2}, \frac 2{t+1}\bigr)$
	such that the function~$s_t$ is strictly increasing on $[0, x^\star(t)]$
	and strictly decreasing on $[x^\star(t), 1/2]$.
\end{lemma}

\begin{proof}
	Define the auxiliary function $h\colon [0, 1]\lra\RR$ by $h(y)=1-ty+ty^{t-1}-y^t$.
	Due to $h''(y)=t(t-1)y^{t-3}(t-2-y)>0$ for $y\in (0, 1]$ this function is strictly convex. 
	Together 
	with $h(0)=1$, $h(1)=0$, and $h'(1)=t(t-3)>0$ this shows that there exists a unique 
	$y^\star=(0, 1)$ such that $h(y^\star)=0$, $h(y)>0$ for $y\in [0, y^\star)$, 
	and $h(y)<0$ for $y\in (y^\star, 1)$. 
	
	Due to 
	\[
		\frac{\mathrm{d}}{\mathrm{d}y}\frac{y+y^t}{(1+y)^{t+1}}
		=
		\frac{h(y)}{(1+y)^{t+2}}
	\]
	it follows that $\frac{y+y^t}{(1+y)^{t+1}}$ is strictly increasing on $[0, y^\star)$ and 
	strictly decreasing on~$(y^\star, 1]$. 
	As~$\frac{2y}{(1+y)^2}$ is strictly increasing on $[0, 1]$ and 
	\[
		s_t\left(\frac{2y}{(1+y)^2}\right)
		=
		\frac{(t+1)(y+y^t)}{(1+y)^{t+1}}\,,
	\]
	it follows that $s_t$ has the desired monotonicity properties 
	for $x^\star(t)=\frac{2y^\star}{(1+y^\star)^2}$.
	
	Next, due to $h(1/t)=t^{2-t}-t^{-t}>0$ we have $y^\star>\frac 1t$ and, consequently, 
	$x^\star(t)> \frac{2t}{(t+1)^2}$. Similarly,
	\[
		h\left(\frac 1{t-1}\right)
		<
		-\frac 1{t-1}+\frac{t}{(t-1)^{t-1}}
		\le 
		\frac{t-(t-1)^2}{(t-1)^3}
		<0
	\]
	yields $y^\star<\frac 1{t-1}$, whence 
	\[
		x^\star(t)
		< 
		\frac{2(t-1)}{t^2}
		<
		\frac 2{t+1}\,. \qedhere
	\]
\end{proof}

\begin{lemma}\label{lem:1908}
	For every integer $t\ge 3$ the function $s_t$ is increasing and concave on $[0, x^\star(t)]$.
\end{lemma}

\begin{proof}
	Our choice of $x^\star(t)$ guarantees that $s_t$ is indeed increasing. So it suffices 
	to show that $s_t$ is concave on the interval $I_t=\bigl[0, \frac 2{t+1}\bigr]$.
	Since 
	\[
		s_t(x)=\frac{t+1}{2^{t-1}}\sum_{0\le n\le (t-1)/2}\binom{t-1}{2n} x(1-2x)^n
	\]
	it suffices to show for every positive integer $n\le (t-1)/2$ that $x(1-2x)^n$
	is concave on $I_t$. This follows immediately from 
	\[
		\frac{\mathrm{d}^2}{\mathrm{d}x^2} x(1-2x)^n
		=
		4n(1-2x)^{n-2}[(n+1)x-1]\,. \qedhere
	\]
 \end{proof}

Our next step is to show $M(S_t,x) = I_2(S_t,x) = s_{t}(x)$ for $x\in[0,x^\star(t)]$.
To this end we use the following result due to Brown and Sidorenko, which is implicit in 
the proof of~\cite{BS94}*{Proposition 2}.

\begin{proposition}[Brown-Sidorenko \cite{BS94}]\label{PROP:BS-merge-incrases-bipartite-graphs}
Let $r$, $s$, $t$, $n$ be positive integers with $r\ge 3$. For every complete $r$-partite 
graph $G$ on $n$ vertices there exists a complete $(r-1)$-partite graph~$G'$ on the same 
vertex set such that $e(G')\le e(G)$ and $N(K_{s, t}, G')\ge N(K_{s, t}, G)$.
\end{proposition}

The proof proceeds by ``merging'' two smallest vertex classes of $G$, i.e., if $V_1, \dots, V_r$
with $|V_1|\le \dots \le |V_r|$ are the vertex classes of $G$, then one constructs $G'$ so as 
to have the
vertex classes $V_1\dcup V_2, V_3, \dots, V_r$. Clearly, $r-2$ iterations of this process lead
to a complete bipartite graph $G''$ such that $V(G'')=V(G)$, $e(G'')\le e(G)$, 
and $N(K_{s, t}, G')\ge N(K_{s, t}, G)$. This shows that for the determination of the 
inducibility of $K_{s, t}$ only complete bipartite host graphs are relevant. This establishes 
the following result on stars.

\begin{theorem}[Brown-Sidorenko \cite{BS94}]\label{THM:BS-inducibility-S_t}
For every integer $t\ge 2$ the inducibility of $S_t$ 
is given by ${\mathrm{ind}}(S_t) = I_2(S_t, x^\star(t))$.
\end{theorem}

We proceed with another simple consequence of 
Proposition~\ref{PROP:BS-merge-incrases-bipartite-graphs}.

\begin{lemma}\label{LEMMA:bipartite-is-optiaml}
If $r, t\ge 2$ are integers and $x\in [0,x^\star(t)]$, then $I_2(S_t,x)\ge I_{r}(S_t,x)$.\end{lemma}

\begin{proof}[Proof of Lemma~\ref{LEMMA:bipartite-is-optiaml}]
Let $y_2 = I_2(S_t,x)$, $y_r = I_r(S_t,x)$ and 
consider an $S_t$-good sequence of complete $r$-partite 
graphs $\left(G_n\right)_{n=1}^{\infty}$ that realizes $(x, y_r)$.
In view of Proposition~\ref{PROP:BS-merge-incrases-bipartite-graphs}
there exists a sequence $\left(G'_n\right)_{n=1}^{\infty}$ of complete bipartite 
graphs such that 
\begin{equation}\label{eq:1838}
	V(G'_n)=V(G_n), 
	\quad 
	e(G'_n)\le e(G_n),
	\quad  
	\text{ and }
	\quad
	N(K_{s, t}, G'_n)\ge N(K_{s, t}, G_n)
\end{equation}
hold for every positive integer $n$.
By passing to a subsequence we may assume that the limits $x' = \lim_{n\to \infty}\rho(G_{n}')$
and $y_2' = \lim_{n\to \infty}\rho(S_t, G_{n}')$ exist. Now~\eqref{eq:1838} 
implies 
\begin{equation}\label{eq:1840}
	x' \le x 
	\quad 
	\text{ and } 
	\quad
	y_2'\ge y_r\,, 
\end{equation}
and as $\left(G_n'\right)_{n=1}^{\infty}$ 
is an $S_t$-good sequence of complete bipartite graphs that realizes~$(x',y_2')$
we have $y_2' \le I_2(S_t, x')$. 
Since $I_2(S_t, \cdot)=s_t(\cdot)$ is increasing on $[0,x^\star(t)]$, the first estimate 
in~\eqref{eq:1840} entails $I_2(S_t,x')\le I_2(S_t,x)$.
So altogether we obtain 
\[
	y_r\le y'_2\le I_2(S_t,x')\le I_2(S_t,x)\,,
\]
which concludes the proof.
\end{proof}

Now we are ready to prove Theorem~\ref{THM:inducibility-St-left-part}.

\begin{proof}[Proof of Theorem~\ref{THM:inducibility-St-left-part}]
The case $t=2$ already being understood in Theorem~\ref{THM:inducibility-S2}
we may assume that $t\ge 3$. It is clear that $I(S_t,x) \ge I_{2}(S_t,x) = s_{t}(x)$ 
holds for $x\in[0,1/2]$ and thus we just need to show $I(S_t,x) \le  s_{t}(x)$ 
for $x\in [0,x^\star(t)]$.
Define $f\colon [0, 1]\longrightarrow\RR$ by 
\begin{align}
f(x) =
\begin{cases}
s_{t}(x) & \quad{\rm for}\quad x\in [0,x^\star(t)] \\
s_{t}\left(x^\star(t)\right) & \quad{\rm for}\quad x\in [x^\star(t),1].
\end{cases}\notag
\end{align}

Lemma~\ref{lem:1908} informs us that $f$ is concave.
Moreover, we have $f(x)\ge M(S_t,x)$ for all $x\in [0, 1]$. Indeed, if $x\in [0,x^\star(t)]$
this follows from Lemma~\ref{LEMMA:bipartite-is-optiaml} and for $x\in [x^\star(t),1]$
we can appeal to Theorem~\ref{THM:BS-inducibility-S_t} instead. 
Summarizing, $f(x)$ is a concave upper bound on $M(S_t, x)$.
Owing to Theorem~\ref{THM:inducibility-multi-partite-all} this 
proves $I(S_t,x) \le f(x) = s_{t}(x)$ for every $x\in [0,x^\star(t)]$.
\end{proof}%THM

\section{Proofs for complete bipartite graphs}\label{SEC:proof-bipartite-graphs}

%%%%%%%%%%%%%%%%%%%%%%%%%%%%%%%%%
In this section we prove Theorems~\ref{THM:feasible-reagion-Kst-upper-bound-x-small}
and~\ref{THM:feasible-reagion-Kst-upper-bound-x-big}.
The upper bound on $I(K_{s, t}, x)$ stated in 
Theorem~\ref{THM:feasible-reagion-Kst-upper-bound-x-small} is an immediate consequence 
of the following result.

\begin{proposition}\label{PROP:complete-bipartite-graphs-decomposition-upper-bound}
	If $t\ge s \ge 2$ are positive integers,
	then for every graph $G$ we have
	\[
		N(K_{s,t}, G) \le \frac{(t-s+1)!}{s!t!} N(S_{t-s+1}, G) \cdot \left(e(G)\right)^{s-1}\,.
	\]
\end{proposition}

\begin{proof}[Proof of Proposition~\ref{PROP:complete-bipartite-graphs-decomposition-upper-bound}]
Notice that for an abstract $K_{s, t}$ the number of 
ordered partitions $V(K_{s, t})=U_1\dcup \dots\dcup U_s$ such that $U_1$ induces 
a star $S_{t-s+1}$
and each of $U_2, \dots, U_s$ induces an edge is $\binom t{t-s+1}(s-1)!s!$. 
This is because there are $s\binom t{t-s+1}$ possibilities for $U_1$; moreover, if $i\in [2, s]$
and $U_1, \ldots, U_{i-1}$ are already fixed, then there are $(s-i+1)^2$ possibilities for $U_i$. 

By double counting it follows that $\binom t{t-s+1}(s-1)!s!N(K_{s,t}, G)$ is at most the 
number of $s$-tuples $(U_1, \dots, U_s)$ of subsets of $G$ such that $G[U_1]\cong S_{t-s+1}$
and $G[U_i]\cong K_2$ for all $i\in [2, s]$, whence 
\[
	  \binom t{t-s+1}(s-1)!s!N(K_{s,t}, G)
	  \le
	  N(S_{t-s+1}, G) \cdot \left(e(G)\right)^{s-1}\,.
\]
Now it remains to observe $\binom t{t-s+1}(s-1)!s!=\frac{s!t!}{(t-s+1)!}$. 
\end{proof}

We remark that this argument is asymptotically optimal if $G$ is a complete bipartite graph. 
More precisely, for $x\le x^\star(t-s+1)$ the sequence $(B(n,x))_{n=1}^\infty$ establishes 
the equality case in Theorem~\ref{THM:feasible-reagion-Kst-upper-bound-x-small}.
This observation concludes the proof of 
Theorem~\ref{THM:feasible-reagion-Kst-upper-bound-x-small}.

In the remainder of this section we show the following explicit version of 
Theorem~\ref{THM:feasible-reagion-Kst-upper-bound-x-big}.

\begin{theorem}\label{THM:4-cycle-induce-upper-bound-x-large}
Every graph $G$ on $n$ vertices with ${xn^2}/{2}$ edges 
satisfies
\begin{align*}%\label{eq:5407}
N(C_4, G) \le  \frac{x(1-x)^2}8 n^4 + 2 n^3\,.
\end{align*}
\end{theorem}

For the proof we need the following well-known result due to Goodman~\cite{Go59},
whose short proof we include for the sake of completeness. 

\begin{proposition}[Goodman \cite{Go59}]\label{PROP:Goodman-bound-constraint}
For every real number $x\in [0,1]$, every positive integer~$n$,
and every graph $G$ on $n$ vertices with $xn^2/2$ edges we have
\begin{align}
\sum_{v\in V(G)}e(v)
\ge  \sum_{v\in V(G)}d(v)^2 - xn^3/2\,, \notag
\end{align}
where $e(v)=e(G[N(v)])$ denotes the number of triangles containing the vertex $v$. 
\end{proposition}

\begin{proof}[Proof of Proposition~\ref{PROP:Goodman-bound-constraint}]
Counting the number of pairs $(u, \{v,w\})\in V(G)\times E(G)$ with $v, w\in N(u)$
in two different ways, we obtain
\[
\sum_{u\in V(G)}e(u)
\ge 
\sum_{vw\in G}\bigl(d(v)+d(w)-n\bigr)
=   
\sum_{v\in V(G)}d(v)^2 - e(G) \cdot n\,. \qedhere
\]
\end{proof}

Goodman's formula has the following consequence, which will assist us in the inductive proof 
of Theorem~\ref{THM:4-cycle-induce-upper-bound-x-large}.

\begin{cor}\label{cor:5409}
	Every graph $G$ with $n$ vertices and $xn^2/2$ edges possesses a vertex $v$ satisfying 
	\[
		e(v)\ge \frac{d(v)^2}{2}+\frac{(1-4x+3x^2)n^2}{4}-\frac{(1-x)^3n^3}{4(n-d(v))}\,.
	\]
\end{cor}

\begin{proof}
	The Cauchy-Schwarz inequality implies $\sum_{v\in V(G)}d(v)^2\ge x^2n^3$ 
	and because of 
	\[
		\sum_{v\in V(G)}(n-d(v))=(1-x)n^2
	\]
	we also have 
	\[
		\sum_{v\in V(G)}\frac 1{n-d(v)}\ge \frac 1{1-x}\,.
	\]
	Consequently, 
	\begin{align*}
		\sum_{v\in V(G)} 
			\Bigl(\frac{d(v)^2}{2}+\frac{(1-4x+3x^2)n^2}{4}-\frac{(1-x)^3n^3}{4(n-d(x))}\Bigr) 
		&\le 
			\sum_{v\in V(G)}\frac{d(v)^2}{2}+\frac{(x^2-x)n^3}{2} \\
		&\le 
			\sum_{v\in V(G)}d(v)^2-xn^3/2\,.
	\end{align*}
	Due to Proposition~\ref{PROP:Goodman-bound-constraint} the result now follows 
	by averaging.  
\end{proof}

\begin{proof}[Proof of Theorem~\ref{THM:4-cycle-induce-upper-bound-x-large}]
	We argue by induction on $n$. The base case $n\le 3$ is clear, for there are no $4$-cycles 
	in graphs with less than four vertices. Now suppose $n\ge 4$ and that our claim
	holds for every graph on $n-1$ vertices.
	
	Given a graph $G$ on $n$ vertices with $xn^2/2$ edges we invoke Corollary~\ref{cor:5409}
	and get a vertex $v\in V(G)$ such that   
	\begin{equation}\label{eq:5421}
		e\ge \frac{d^2}{2}+\frac{(1-4x+3x^2)n^2}{4}-\frac{(1-x)^3n^3}{4(n-d)}\,,
	\end{equation}
	where $d=d(v)$ and $e=e(v)$. We contend that 
	\begin{equation}\label{eq:5422}
		N(C_4 , G)\le N(C_4, G-v)+(d^2/2-e)(n-d)\,,
	\end{equation}
	or, in other words, that there are at most $(d^2/2-e)(n-d)$ induced copies 
	of $K_{2, 2}$ in $G$ which contain the vertex $v$. The reason for this is that each 
	such copy contains a pair of non-adjacent members of $N(v)$ and a fourth vertex belonging 
	to $V(G)\sm N(v)$. Clearly there are at most
	$d^2/2-e$ possibilities for such a non-adjacent pair and at most $n-d$ possibilities 
	for the fourth vertex.
	
	\begin{claim}\label{clm:5429}
		We have 
		\[
			8N(C_4, G-v)\le x(1-x)^2(n^4-4n^3)+2(xn-d)(1-4x+3x^2)n^2+16n^3\,.
		\]
	\end{claim} 
	
	\begin{proof}
		The induction hypothesis yields
		\begin{equation}\label{eq:5947}
			8N(C_4, G-v)\le x'(1-x')^2(n-1)^4+16(n-1)^3\,,
		\end{equation}
		where $x'$ is defined by 
		\[
			x'=\frac{2|E(G-v)|}{(n-1)^2}=\frac{xn^2-2d}{(n-1)^2}\,.
		\]

		The function $h(x)=x(1-x)^2$ has derivatives $h'(x)=1-4x+3x^2$ and $h''(x)=-4+6x$.
		Therefore we have $\|h'\|_{[0, 1]}=1$ and $\|h''\|_{[0, 1]}=4$, 
		where $\|\cdot\|_{[0, 1]}$ denotes the supremum norm with respect to the unit interval.
		So Taylor's formula and~\eqref{eq:5947} imply 
		\begin{align*}
			8N(C_4, G-v)
			\le 
			x(1-x)^2(n-1)^4&+(1-4x+3x^2)(x'-x)(n-1)^4 \\ 
			&+2(x'-x)^2(n-1)^4+16(n-1)^3\,.
		\end{align*}
		Here 
		\[
			x(1-x)^2(n-1)^4\le x(1-x)^2(n^4-4n^3+6n^2)\le x(1-x)^2(n^4-4n^3)+n^2
		\]
		and due to 
		\begin{equation}\label{eq:5006}
			x'-x=\frac{(2n-1)x-2d}{(n-1)^2}
		\end{equation}
		we have $2(x'-x)^2(n-1)^4 = 2\big|(2n-1)x-2d\big|^2\le 8n^2$. 
		For these reasons it suffices to establish 
		\begin{equation}\label{eq:5959}
			(1-4x+3x^2)(x'-x)(n-1)^4
			\le
			2(xn-d)(1-4x+3x^2)n^2+7n^2\,.
		\end{equation}
		Now the triangle inequality yields 
		\begin{align*}
			\big|(x'-x)(n-1)^4 & -2(xn-d)n^2\big| \\
			&\le 
			\big|(x'-x)(n-1)^2-2(xn-d)\big|(n-1)^2+2|xn-d|\bigl(n^2-(n-1)^2\bigr) \\
				&\overset{\eqref{eq:5006}}{\le}
			x(n-1)^2+4n^2
				\le 
			5n^2
		\end{align*}
		and together with $\|h'\|_{[0, 1]}=1$ this proves~\eqref{eq:5959}.
		Thereby Claim~\ref{clm:5429} is proved.
	\end{proof}
	
	Now combining~\eqref{eq:5421},~\eqref{eq:5422}, and Claim~\ref{clm:5429} we obtain
	\begin{align*}
		8N(C_4, G)
		&\le 
		x(1-x)^2(n^4-4n^3)+2(xn-d)(1-4x+3x^2)n^2+16n^3 \\
		&\phantom{xxxxxxxx}-2(1-4x+3x^2)n^2(n-d)+2(1-x)^3 n^3 \\
		&=
		x(1-x)^2n^4+16n^3\,,
	\end{align*}
	as desired. 
\end{proof}

%%%%%%%%%%%%%%%%%%%%%%%%%%%%%%%%%%%%%%%%%%
\section{Concluding remarks}\label{SEC:remarks}

\subsection{General questions} 
As the example $Q = K_3 + \overline{K}_{3}$ shows, for a quantum graph~$Q$ 
the function $I(Q,x)$ can have at least two global maxima.
We do not know whether this is possible for single graphs $F$ as well. 

\begin{problem}\label{PROB:two-global-max}
Does there exist a graph $F$ such that the function $I(F,x)$ has at least two global maxima?
\end{problem}

Two questions of a similar flavor are as follows.

\begin{problem}\label{PROB:interval-global-max}
Does there exist a graph $F$ such that for some nontrivial interval $J$ we 
have $I(F,x) = {\mathrm{ind}}(F)$ for all $x\in J$?
\end{problem}

\begin{problem}\label{PROB:local-max-min}
Does there exist a graph $F$ such that the function $I(F,x)$ has a nontrivial local maximum (minimum)?
\end{problem}

Recall that for a self-complementary graph $F$ the function $I(F,x)$ is symmetric 
around $x = 1/2$. One may thus wonder whether some appropriate self-complementary 
graph $F$ yields an affirmative solution to Problem~\ref{PROB:two-global-max}. 
This approach leads to the following question.

\begin{problem}\label{PROB:self-complementary-graph}
Let $F$ be a self-complementary graph.
Is it true that $I(F,x)={\mathrm{ind}}(F)$ holds if and only if $x = 1/2$?
\end{problem}

\subsection{Problems for specific graphs}
The smallest problem left open by our results on stars in Section~\ref{SEC:proof-stars}
is to determine $I(S_3, x)$ for $x\in [1/2, 1]$. In an interesting contrast to the case 
$S_2=K_3^-$ one can show that the clique density construction (see Construction~\ref{con:cdt})
is not extremal for this problem. For $x\in [4\sqrt{2}-5, 1]$ the best construction we are 
aware of
is the complement of a clique of order $\lfloor (1-x)^{1/2}n\rfloor$, which leads 
to the bound 
\begin{equation}\label{eq:9999}
	I(S_3, x)\ge 4(1-(1-x)^{1/2})(1-x)^{3/2}\,.
\end{equation}
For $x\in [0.91, 0.93]$ we have a complicated argument based on the results in~\cite{RW18}
which shows that equality holds in~\eqref{eq:9999}. 
In the range $x\in [1/2, 4\sqrt{2}-5)$ the complement of two 
disjoint cliques of order $\lfloor((1-x)/2)^{1/2} n\rfloor$ shows that $I(S_3, x)$
is strictly larger than the right side of~\eqref{eq:9999}. We hope to return to this 
problem in the near future. 

Finally, we would like to emphasize Conjecture~\ref{conj:cdt} again: 
Is it true that for $x\in [1/2, 1]$ the graphs in Construction~\ref{con:cdt}
minimizing the triangle density maximize the induced~$C_4$ density?
   
\subsection*{Acknowledgement} We would like to thank both referees for helpful
suggestions.


\begin{bibdiv}
\begin{biblist}
\bib{AHKT}{article}{
   author={Alon, Noga},
   author={Hefetz, Dan},
   author={Krivelevich, Michael},
   author={Tyomkyn, Mykhaylo},
   title={Edge-statistics on large graphs},
   journal={Combin. Probab. Comput.},
   volume={29},
   date={2020},
   number={2},
   pages={163--189},
   issn={0963-5483},
   review={\MR{4079631}},
   doi={10.1017/s0963548319000294},
}

\bib{Bo76}{article}{
   author={Bollob\'{a}s, B\'{e}la},
   title={On complete subgraphs of different orders},
   journal={Math. Proc. Cambridge Philos. Soc.},
   volume={79},
   date={1976},
   number={1},
   pages={19--24},
   issn={0305-0041},
   review={\MR{424614}},
   doi={10.1017/S0305004100052063},
}
		
\bib{BS94}{article}{
   author={Brown, Jason I.},
   author={Sidorenko, Alexander},
   title={The inducibility of complete bipartite graphs},
   journal={J. Graph Theory},
   volume={18},
   date={1994},
   number={6},
   pages={629--645},
   issn={0364-9024},
   review={\MR{1292981}},
   doi={10.1002/jgt.3190180610},
}
		
\bib{BL16}{article}{
   author={Bubeck, S\'{e}bastien},
   author={Linial, Nati},
   title={On the local profiles of trees},
   journal={J. Graph Theory},
   volume={81},
   date={2016},
   number={2},
   pages={109--119},
   issn={0364-9024},
   review={\MR{3433633}},
   doi={10.1002/jgt.21865},
}

\bib{Er62}{article}{
   author={Erd\H{o}s, P.},
   title={On the number of complete subgraphs contained in certain graphs},
   language={English, with Russian summary},
   journal={Magyar Tud. Akad. Mat. Kutat\'{o} Int. K\"{o}zl.},
   volume={7},
   date={1962},
   pages={459--464},
   issn={0541-9514},
   review={\MR{151956}},
}

\bib{EL15}{article}{
   author={Even-Zohar, Chaim},
   author={Linial, Nati},
   title={A note on the inducibility of 4-vertex graphs},
   journal={Graphs Combin.},
   volume={31},
   date={2015},
   number={5},
   pages={1367--1380},
   issn={0911-0119},
   review={\MR{3386015}},
   doi={10.1007/s00373-014-1475-4},
}

\bib{FS}{article}{
   author={Fox, Jacob},
   author={Sauermann, Lisa},
   title={A completion of the proof of the edge-statistics conjecture},
   journal={Adv. Comb.},
   date={2020},
   pages={Paper No. 4, 52},
   review={\MR{4125345}},
   doi={10.19086/aic.12047},
}
				
\bib{Go59}{article}{
   author={Goodman, A. W.},
   title={On sets of acquaintances and strangers at any party},
   journal={Amer. Math. Monthly},
   volume={66},
   date={1959},
   pages={778--783},
   issn={0002-9890},
   review={\MR{107610}},
   doi={10.2307/2310464},
}
		
		
\bib{HN19}{article}{
   author={Hatami, Hamed},
   author={Norin, Sergey},
   title={On the boundary of the region defined by homomorphism densities},
   journal={J. Comb.},
   volume={10},
   date={2019},
   number={2},
   pages={203--219},
   issn={2156-3527},
   review={\MR{3912211}},
   doi={10.4310/JOC.2019.v10.n2.a1},
}
		
		
\bib{HN11}{article}{
   author={Hatami, Hamed},
   author={Norine, Serguei},
   title={Undecidability of linear inequalities in graph homomorphism
   densities},
   journal={J. Amer. Math. Soc.},
   volume={24},
   date={2011},
   number={2},
   pages={547--565},
   issn={0894-0347},
   review={\MR{2748400}},
   doi={10.1090/S0894-0347-2010-00687-X},
}

\bib{Hirst14}{article}{
   author={Hirst, James},
   title={The inducibility of graphs on four vertices},
   journal={J. Graph Theory},
   volume={75},
   date={2014},
   number={3},
   pages={231--243},
   issn={0364-9024},
   review={\MR{3153118}},
   doi={10.1002/jgt.21733},
}

\bib{HLNPS14}{article}{
   author={Huang, Hao},
   author={Linial, Nati},
   author={Naves, Humberto},
   author={Peled, Yuval},
   author={Sudakov, Benny},
   title={On the 3-local profiles of graphs},
   journal={J. Graph Theory},
   volume={76},
   date={2014},
   number={3},
   pages={236--248},
   issn={0364-9024},
   review={\MR{3200287}},
   doi={10.1002/jgt.21762},
}

%\bib{Ka39}{book}{
%   author={Karush, William},
%   title={Minima of functions of several variables with inequalities as side
%   conditions},
%   note={Thesis (SM)--The University of Chicago},
%   publisher={ProQuest LLC, Ann Arbor, MI},
%   date={1939},
%   pages={25},
%   review={\MR{2936770}},
%}		
		
\bib{Ka68}{article}{
   author={Katona, G.},
   title={A theorem of finite sets},
   conference={
      title={Theory of graphs},
      address={Proc. Colloq., Tihany},
      date={1966},
   },
   book={
      publisher={Academic Press, New York},
   },
   date={1968},
   pages={187--207},
   review={\MR{0290982}},
}
		
%\bib{KS06}{article}{
%   author={Krivelevich, M.},
%   author={Sudakov, B.},
%   title={Pseudo-random graphs},
%   conference={
%      title={More sets, graphs and numbers},
%   },
%   book={
%      series={Bolyai Soc. Math. Stud.},
%      volume={15},
%      publisher={Springer, Berlin},
%   },
%   date={2006},
%   pages={199--262},
%   review={\MR{2223394}},
%   %doi={10.1007/978-3-540-32439-3_10},
%}
	
\bib{Kru63}{article}{
   author={Kruskal, Joseph B.},
   title={The number of simplices in a complex},
   conference={
      title={Mathematical optimization techniques},
   },
   book={
      publisher={Univ. of California Press, Berkeley, Calif.},
   },
   date={1963},
   pages={251--278},
   review={\MR{0154827}},
}

\bib{KST19}{article}{
   author={Kwan, Matthew},
   author={Sudakov, Benny},
   author={Tran, Tuan},
   title={Anticoncentration for subgraph statistics},
   journal={J. Lond. Math. Soc. (2)},
   volume={99},
   date={2019},
   number={3},
   pages={757--777},
   issn={0024-6107},
   review={\MR{3977889}},
   doi={10.1112/jlms.12192},
}

%\bib{KT51}{article}{
%   author={Kuhn, H. W.},
%   author={Tucker, A. W.},
%   title={Nonlinear programming},
%   conference={
%      title={Proceedings of the Second Berkeley Symposium on Mathematical
%      Statistics and Probability, 1950},
%   },
%   book={
%      publisher={University of California Press, Berkeley and Los Angeles},
%   },
%   date={1951},
%   pages={481--492},
%   review={\MR{0047303}},
%}

\bib{LPSS}{article}{
	author={Liu, Hong},
	author={Pikhurko, Oleg},
	author={Sharifzadeh, Maryam},
	author={Staden, Katherine},
	title={Stability from graph symmetrisation arguments with applications to inducibility},
	eprint={2012.10731},
	%note={Submitted},
}

\bib{LPS}{article}{
   author={Liu, Hong},
   author={Pikhurko, Oleg},
   author={Staden, Katherine},
   title={The exact minimum number of triangles in graphs with given order
   and size},
   journal={Forum Math. Pi},
   volume={8},
   date={2020},
   pages={e8, 144},
   review={\MR{4089395}},
   doi={10.1017/fmp.2020.7},
}
		
\bib{LM21A}{article}{
   author={Liu, Xizhi},
   author={Mubayi, Dhruv},
   title={The feasible region of hypergraphs},
   journal={J. Combin. Theory Ser. B},
   volume={148},
   date={2021},
   pages={23--59},
   issn={0095-8956},
   review={\MR{4193065}},
   doi={10.1016/j.jctb.2020.12.004},
}

\bib{MMNT}{article}{
   author={Martinsson, Anders},
   author={Mousset, Frank},
   author={Noever, Andreas},
   author={Truji\'{c}, Milo\v{s}},
   title={The edge-statistics conjecture for $\ell\ll k^{6/5}$},
   journal={Israel J. Math.},
   volume={234},
   date={2019},
   number={2},
   pages={677--690},
   issn={0021-2172},
   review={\MR{4040841}},
   doi={10.1007/s11856-019-1929-8},
}

\bib{Niki11}{article}{
   author={Nikiforov, V.},
   title={The number of cliques in graphs of given order and size},
   journal={Trans. Amer. Math. Soc.},
   volume={363},
   date={2011},
   number={3},
   pages={1599--1618},
   issn={0002-9947},
   review={\MR{2737279}},
   doi={10.1090/S0002-9947-2010-05189-X},
}
	
\bib{Olpp96}{article}{
   author={Olpp, Dieter},
   title={A conjecture of Goodman and the multiplicities of graphs},
   journal={Australas. J. Combin.},
   volume={14},
   date={1996},
   pages={267--282},
   issn={1034-4942},
   review={\MR{1424340}},
}

\bib{PG75}{article}{
   author={Pippenger, Nicholas},
   author={Golumbic, Martin Charles},
   title={The inducibility of graphs},
   journal={J. Combinatorial Theory Ser. B},
   volume={19},
   date={1975},
   number={3},
   pages={189--203},
   issn={0095-8956},
   review={\MR{401552}},
   doi={10.1016/0095-8956(75)90084-2},
}
	
\bib{Ra08}{article}{
   author={Razborov, Alexander A.},
   title={On the minimal density of triangles in graphs},
   journal={Combin. Probab. Comput.},
   volume={17},
   date={2008},
   number={4},
   pages={603--618},
   issn={0963-5483},
   review={\MR{2433944}},
   doi={10.1017/S0963548308009085},
}

\bib{Re16}{article}{
   author={Reiher, Chr.},
   title={The clique density theorem},
   journal={Ann. of Math. (2)},
   volume={184},
   date={2016},
   number={3},
   pages={683--707},
   issn={0003-486X},
   review={\MR{3549620}},
   doi={10.4007/annals.2016.184.3.1},
}
		
\bib{RW18}{article}{
   author={Reiher, Chr.},
   author={Wagner, Stephan},
   title={Maximum star densities},
   journal={Studia Sci. Math. Hungar.},
   volume={55},
   date={2018},
   number={2},
   pages={238--259},
   issn={0081-6906},
   review={\MR{3813354}},
   doi={10.1556/012.2018.55.2.1395},
}

\bib{ST98}{article}{
   author={Schelp, Richard H.},
   author={Thomason, Andrew},
   title={A remark on the number of complete and empty subgraphs},
   journal={Combin. Probab. Comput.},
   volume={7},
   date={1998},
   number={2},
   pages={217--219},
   issn={0963-5483},
   review={\MR{1617934}},
   doi={10.1017/S0963548397003234},
}

%\bib{Thom87}{article}{
%   author={Thomason, Andrew},
%   title={Pseudorandom graphs},
%   conference={
%      title={Random graphs '85},
%      address={Pozna\'{n}},
%      date={1985},
%   },
%   book={
%      series={North-Holland Math. Stud.},
%      volume={144},
%      publisher={North-Holland, Amsterdam},
%   },
%   date={1987},
%   pages={307--331},
%   review={\MR{930498}},
%}

\bib{Thom89}{article}{
   author={Thomason, Andrew},
   title={A disproof of a conjecture of Erd\H{o}s in Ramsey theory},
   journal={J. London Math. Soc. (2)},
   volume={39},
   date={1989},
   number={2},
   pages={246--255},
   issn={0024-6107},
   review={\MR{991659}},
   doi={10.1112/jlms/s2-39.2.246},
}
		
\bib{TU41}{article}{
   author={Tur\'{a}n, Paul},
   title={Eine Extremalaufgabe aus der Graphentheorie},
   language={Hungarian, with German summary},
   journal={Mat. Fiz. Lapok},
   volume={48},
   date={1941},
   pages={436--452},
   issn={0302-7317},
   review={\MR{18405}},
}
		
				
		

			
\end{biblist}
\end{bibdiv}
\end{document}